\documentclass[12pt]{article}

\usepackage{amssymb,amsmath,amsthm,sectsty,url,mathrsfs,graphicx}
\usepackage[letterpaper,hmargin=1.0in,vmargin=1.0in]{geometry}
\usepackage[colorlinks,linkcolor=black,citecolor=black,filecolor=black,urlcolor=black]{hyperref}
\usepackage{color}

\sectionfont{\large}
\subsectionfont{\normalsize}
\numberwithin{equation}{section}

\DeclareMathSymbol{\leqslant}{\mathalpha}{AMSa}{"36} 
\DeclareMathSymbol{\geqslant}{\mathalpha}{AMSa}{"3E} 
\DeclareMathSymbol{\eset}{\mathalpha}{AMSb}{"3F}     
\renewcommand{\le}{\;\leqslant\;}                   
\renewcommand{\ge}{\;\geqslant\;}                   

\newtheorem{maintheorem}{Theorem}
\newtheorem{theorem}{Theorem}[section]
\newtheorem{lemma}[theorem]{Lemma}
\newtheorem{proposition}[theorem]{Proposition}
\newtheorem{corollary}[theorem]{Corollary}
\newtheorem{fact}[theorem]{Fact}
\newtheorem{definition}[theorem]{Definition}
\newtheorem{remark}[theorem]{Remark}
\newtheorem{conjecture}[theorem]{Conjecture}

\newtheorem{example}[theorem]{Example}
\newtheorem{open}[theorem]{Open Problem}
\newcommand{\PP}{{\mathbb P}}
\def\E{\mathop{\mathbb E}}

\newcommand{\as}{\mathrm{a.s.}}
\newcommand{\whp}{\mathrm{w.h.p.}}

\newcommand{\Pois}{\mathrm{Pois}}

\def\eps{\varepsilon}
\renewcommand{\epsilon}{\varepsilon}

\newcommand{\lac}{\lambda_{\mathrm{c}}}

\renewcommand{\Pr}{ \mathrm P}

\newcommand{\gd}{\delta}

\newcommand{\la}{\lambda}

\newcommand{\SC}{ \mathrm{SC} }

\newcommand{\F}{{\cal F}}

\newcommand{\simt}{\stackrel{t}{\sim}}

\newcommand{\meett}{\stackrel{t}{\leftrightarrow}}

\newcommand{\sfrac}[2]{\mbox{\small $\frac{#1}{#2}$}}
\newcommand{\ssfrac}[2]{\mbox{\footnotesize $\frac{#1}{#2}$}}
\newcommand{\half}{\ssfrac{1}{2}}

\let\phi=\varphi

\newcommand{\N}{\mathbb N}
\newcommand{\R}{\mathbb R}
\newcommand{\Z}{\mathbb Z}

\newcommand{\W}{\mathcal W}

\newcommand{\Ind}[1]{\mathbf{1}\{#1\}}

\usepackage[normalem]{ulem}

\begin{document}

\title{Rapid social connectivity}
\author{Itai Benjamini
\thanks{
Department of Mathematics, Weizmann Institute
of Science, Rehovot 76100,
Israel. E-mail: {\tt itai.benjamini@weizmann.ac.il}.}
\and Jonathan Hermon
\thanks{
University of Cambridge, Cambridge, UK. E-mail: {\tt jonathan.hermon@statslab.cam.ac.uk}. Financial support by
the EPSRC grant EP/L018896/1.}
}
\date{}
\maketitle

\begin{abstract}
Given a  graph $G=(V,E)$,  consider Poisson($ |V|$) walkers performing independent lazy simple random walks on $G$ simultaneously, where the initial position of each walker is chosen independently with probability proportional to the degrees.
When two walkers visit the same vertex at the same time they are declared to be
acquainted. The social connectivity time $\SC(G)$ is defined as the first time in which there
is a path of acquaintances between every pair of walkers. It is shown that when the average degree of $G$ is $d$, with high probability \[ c\log |V| \le \SC(G) \le C d^{1+5 \cdot 1_{G \text{ is not regular}} } \log^3 |V|.\]
When $G$ is regular the  lower bound is improved to $\SC(G) \ge \log |V| -6 \log \log  |V| $, with high probability.
 We determine $\SC(G)$ up to a constant factor in the cases that $G$ is an expander and when it is the $n$-cycle. \end{abstract}

\newpage


\section{Introduction}
Consider the following model for a social network which we call the \textbf{random walks social
network model} or for short, the \emph{social network model}. We write \textbf{SN} as a shorthand for social network. Let $G=(V,E)$
be a finite connected graph,
which is the underlying graph of the SN model. 
In the model there are walkers performing independent lazy simple random
walks (\textbf{LSRW}) on $G$, defined as follows.  If a
walker's current position is $v$, then the walker either stays put w.p.~$1/2$, or moves to one of the neighbors of $v$ w.p.~$\frac{1}{2d_{v}}$, where $d_v$ is the degree of $v$.
The walkers perform their LSRWs simultaneously (i.e.~at each time unit they
all perform one step, which may be a lazy step).

 At time $0$ the number of walkers at vertex $v$ is set to be $N_{v}$, where $(N_{v})_{v \in V}$ are independent random variables, where for each vertex $v$, we have that $N_v \sim \mathrm{Poisson}( \bar \pi_v)$  where $\bar \pi := |V| \pi $ and $\pi_v:= d_v/(2|E|)$ is the stationary distribution of the random walk on $G$. 
 
\medskip

A directed graph $G$ is called \emph{Eulerian} if for every vertex $v$ the indegree of $v$ equals its outdegree, denoted (unambiguously) by $d_{v}$. One can define LSRW on $G=(V,E)$ in an analogous manner to the undirected case (e.g., \cite{PS}).  Its stationary distribution is  $\pi_v:= d_v/|E|$ (here $E$ is the set of directed edges). We say that a directed graph $G=(V,E)$ is \emph{connected} if for all $u,v \in V$ there is a directed path in $G$ from $u$ to $v$.\footnote{This is sometimes referred to as being strongly connected.} We
say that an Eulerian graph $G=(V,E)$ is regular when $d_v=d_u$ for all $u,v \in V$. We denote a directed edge from $u $ to $v$ by $(u,v)$ and an undirected edge by $\{u,v\}$.

 We shall prove our main results in the setup that $G$ is a connected Eulerian digraph. Note that any graph can be thought of as an Eulerian digraph in which each edge appears with both orientations. However, the examples we consider shall all be graphs. Namely, the $n$-cycle and regular expanders.

\medskip

We say that two
walkers $w,w'$ have \emph{met by time} $t$, which we denote by $w \meett w' $, if
there exists $t_{0} \le t$ such that they have the same position at time
$t_{0}$. After two walkers meet they
continue their walks independently without coalescing. ``Meeting
by time $t$" is a symmetric relation. It  induces the equivalence relation of \emph{having
a path of acquaintances
by time} $t$,  denoted by $\simt$, defined as follows. Two walkers $a$ and $b$ have a path of acquaintances
by time $t$ iff  there exist $k\in \N$ and walkers $a=c_{0},c_{1},\ldots,c_{k},c_{k+1}=b$ such that $c_{i} \meett c_{i+1} $,
 for all $0 \le i \le k$. Note that we are not requiring the sequence of times in which the walkers met to be non-decreasing, which is the main difference between the SN model and some existing models for spread of rumor/infection (e.g.~the $A+B \to 2B$ model \cite{kesten2005spread} and the Frog model \cite{telcs1999branching,alves2002phase,alves2002shape,popov2001frogs,hoffman2014frog}, see \S\ref{s:related} for a description of these models and a discussion about their relationship to the SN model). Consequently, the SN model often evolves  faster than such models. However,  Conjecture \ref{con:polylog} suggests one sense in which  these models   might also evolve rapidly. 

\medskip

We are interested
in the coalescence process of the equivalence classes
of  $\simt$ (as $t$ varies). In particular, we are interested in the  \textbf{social connectivity time} $\SC(G)$, defined as the minimal time in which there is only one class (i.e.~every pair of walkers have a path of acquaintances between them).
\section{Our results}
\subsection{General graphs}
Our main result is 
\begin{maintheorem}
\label{thm: 1}
Let $G=(V,E)$ be a connected $n$-vertex Eulerian digraph of average degree $d=\sfrac{|E|}{n}$. Denote $r_{*}:= \sfrac{ \min_{u\in V}d_u}{d}  $. Then there exist absolute constants\footnote{We use $C,C',C_0,C_1,K,M\ldots$ (resp.~$\delta,\epsilon, c,c',c_0,c_1,\ldots $) to denote positive absolute constants which are sufficiently large (resp.~small) to ensure that a certain inequality holds. Similarly, we use $C_{d}$ (resp.~$c_{d}$) to refer to sufficiently large (resp.~small) positive constants, whose value depends on the parameters appearing in subscript. Different appearances of the same constant at different places may refer to different numeric values.} $C,M,c>0$ such that
\begin{equation}
\label{eq: upper}
\Pr[ \SC(G) > C d^{1+ 5 \xi  } (\log n)^{3} ] \le M/n, \quad \text{where }  \xi:=1_{G \text{ is not regular}}.
\end{equation}
\begin{equation}
\label{eq: lower}
\forall \, \alpha \in (0,1),\, \exists \, c_{\alpha}>0, \quad \Pr[ \SC(G) \le c_{\alpha} \log n  ] \le 3e^{-cr_{*}^2n^{\alpha}} . 
\end{equation}
\end{maintheorem}
Note that \eqref{eq: lower} is meaningful as long as $r_* \ge |V|^{\beta-0.5}$ for some $\beta >0$. Theorem \ref{thm: lower2} gives an  extension  covering in particular the case $r_{*} \ge |V|^{\beta-1}$ for some $\beta > 0$.
When $G$ is a regular graph we have the following improved lower bounds on $\SC(G) $.
\begin{maintheorem}
\label{thm: reg}
There exist $c,C>0$ so that for every finite connected regular graph   $G=(V,E)$,  \begin{equation}
 \label{e:reglower1}
\Pr[\SC(G)< \log |V| -6 \log \log |V| ] \le C/\log |V|. \end{equation}
 \begin{equation}
 \label{e:reglower2}
 \Pr[\SC(G)< s_{*}(G) ] \le C \exp(- c \sqrt{|V|} (\log |V|)^{-4} ),  \end{equation}
 where $P$ is the transition matrix of the LSRW, \[\kappa_t:=\min_{v \in V}\sum_{i=0}^tP^{2i}(v,v)  \quad \text{and} \quad   s_{*}(G):=\min \{t: \sfrac{ t+1}{\kappa_{t+1}} > \sfrac{1}{32}  \log |V| \}.\]
\end{maintheorem}
Note that $s_*(G) >  \sfrac{1}{32}  \log |V| $. We include  \eqref{e:reglower1}, despite this, as it is sharp, up to the value of the sub-leading term, and the failure probability, see the discussion below about extremality of the complete graph. In some cases $s_*(G) \ge C( \log |V|)^{2} $. For instance, for the $n$-cycle $\mathrm{C}_n$ we have that $P^{2i}(v,v)= \Theta(\sfrac{1}{\sqrt{i+1}}) $ and so $s_*(\mathrm{C}_n) = \Theta((\log n)^2)$.\footnote{We write $o(1)$ for terms which vanish as $n \to \infty$ (or some other index, which is clear from context). We write $f_n=o(g_n)$ or $f_n \ll g_n$ if $f_n/g_n=o(1)$. We write $f_n=O(g_n)$ and $g_n=\Omega(f_n) $ if there exists a constant $C>0$ such that $|f_n| \le C |g_n|$ for all $n$. We write  $f_n=\Theta(g_n)$ if  $f_n=O(g_n)$ and  $g_n=O(f_n)$. We write $f_{n} \le O_d(g_{n})$ to indicate that the constant $C$ above may depend on $d$.} Hence by \eqref{e:reglower2} $\Pr[\SC(\mathrm{C}_n)< c (\log n)^2]=o(1) $. This is not surprising as initially there will be intervals of length $\Omega(\log n) $ which will not be occupied by a single particle. This is why in terms of the lower bound on $\SC(\mathrm{C}_n)$ it is more interesting to study the case that initially there is one walker per site (as done in Theorem \ref{thm: cycle}). Conversely, there exists some $C>0$ such that the probability that $\SC(\mathrm{C}_n) > C s_*(\mathrm{C}_n) \ge C'(\log n)^2   $ is at most $C''/n$ (see \S\ref{s:introcycle}). We conjecture that there exists some universal constant $C>0$ such that for every connected finite vertex transitive graph $G$ (Definition \ref{def: VT})  $\sfrac{1}{C} \le \sfrac{\mathbb{E}[\SC(G)]}{s_{*}(G)+1} \le C $.   
\subsection{Some remarks concerning our main results}
\textbf{Degree dependence:} We now discuss the degree dependence in \eqref{eq: upper} and also a refinement of \eqref{eq: upper}.
Example \ref{rem: extermal} (resp.\ \ref{ex:nonregular})  describes for some absolute constants $c_0,c_1,C_0>0$ for all sufficiently large $n$ and all $C_{0} \le d_n \le c_0n $  a  family of graphs $G_n=(V_{n},E_{n})$ such that $n \le |V_n| \le C_{0}n $ which are  $d_{n}$-regular (resp.\ of average degree $d_n$) for which  $\Pr[ \SC(G_n)\ge c_{1} d_{n} \log ^{2}(|V_n|/d_{n})]> c_{1}$ (resp.\   $\Pr[ \SC(G_n)\ge c_{1} d_{n}^2 \log ^{2}(|V_n|/d_{n})]> c_{1}$). This shows that the linear dependence on $d$  in \eqref{eq: upper} is optimal for regular $G$ and that we can have $\mathbb{E}[ \SC(G)] \ge n  $ for an $n$-vertex regular $G$ and $\mathbb{E}[ \SC(G)] \ge n^{2}$ for non-regular $G$.

 If $G$ is $d$-regular and in addition for every set $A \subset V$ such that $|A| \le  \sfrac{2}{3} (|V|+|V|^{2/3})$, we have that \[|\{(u,v) \in E : u \in A,v \notin A \text{ or }u \notin A,v \in A \}| \ge c_0 d,\]  then the degree dependence in  \eqref{eq: upper} can be eliminated and \eqref{eq: upper} can be improved to \[\Pr[ \SC(G) > C c_0^{-2} \log^3|V| ] \le M/|V|,\] for some  absolute constants $C,M$ (see Proposition \ref{prop: main}, part (iii)).   

Now consider the case that $G$ is a connected $n$-vertex Eulerian digraph of average degree $d$. Let $\xi:=  1_{G \text{ is not regular}} $. From the proof of Theorem \ref{thm: 1} it follows that  we can replace  $ Cd^{1+ 5 \xi } (\log n)^3 $ in \eqref{eq: upper} by $C' t_*(G) d^{1+3 \xi }  \log n$, for some $t_*(G) \le \bar C  d^{2\xi }   (\log n)^2$.  Namely, \[t_{*}(G):= \inf \{t: \max_{x,y}|P^t(x,y)-\pi_y |\le d_{y}/(3 \cdot 2^{9}d \log n) \} ,\] where $P$ is the transition matrix of the LSRW on $G$. By Fact \ref{fact: decay} indeed  $ t_{*}(G)\le \bar C d^{2\xi} (\log n)^2$. 

For the cycle $t_{*}=\Theta((\log n)^2)$. This shows that  \eqref{eq: upper} is sharp up to a factor $O( \log n )$, when $d=O(1)$.  We believe that the cycle is in some sense extremal for the SN model (see Remark \ref{rem: extermal}, Theorem \ref{thm: cycle} and Conjecture \ref{con: cycle} for more details on this point). 
%
%
%

The next remark discusses the case in which the particle density is considered as a parameter of the model.

\medskip

\noindent
\textbf{Allowing the particle density to vary:}
One may consider the case in which the density of walkers at each site is multiplied by some $\lambda >0 $ (i.e.\ for all $v \in V $ the number of particles initially occupying it  $N_v$ has a $ \mathrm{Poisson}( \la \bar  \pi_v)$  distribution). Let $\xi:=1_{G \text{ is not regular}}$.  For $C_{0}\frac{ \log |V|}{|V|}d^\xi  \le \lambda \le 1$ only minor adaptations to the arguments from \S\ref{s: upper} are necessary in order to show that  \begin{equation*}
 \label{e:laintro}
 \whp  \quad \SC(G) \le  C_d \lambda^{-4} ( \log |V|)^3,  
\end{equation*}
(here we allow $\la$ to depend on $|V|$, provided that $\la \ge C\frac{ \log |V|}{|V|}d^\xi     $). With  additional work this can be improved to $\whp $ $\SC(G) \le  C d^{1+5 \xi } \lambda^{-2} ( \log |V|)^3 $.  When  $\lambda \in [|V|^{-1+ \eps } ,1] $ for some $\eps \in (0,1)$, this is optimal up to an $O_{d,\eps}(\log |V|)$ factor as described in Remark \ref{rem:lacycle}.   

When $G$ is vertex-transitive (see Definition \ref{def: VT}) we conjecture that \[\mathbb{E}_{\lambda}[\SC(G) \mid \text{there are at least two walkers}]\] is monotone decreasing in $\la$ (see Conjecture \ref{con: monotone} for a slightly different formulation). 

%

It is not hard to extend the proofs of \eqref{eq: lower} and \eqref{e:reglower1} to the case that the particle density at each site is multiplied by some $\lambda \in [|V|^{- \eps } ,1] $ for some $\eps \in (0,1/2)$. In this case, the lower bound on $\SC(G)$ is increased by a factor of $\la^{-1}$. Similarly, one can extend \eqref{e:reglower2} as follows: if $G=(V,E)$ is a connected $n$-vertex regular graph and $\lambda \in [n^{- \eps } ,1] $ for some $\eps \in (0,1/16)$ then $\whp$ $\SC(G) \ge s_{*}(G,\la):=\min \{t: \sfrac{ t+1}{\kappa_{t+1}} > \sfrac{1}{32 \la }  \log n \} $.   

\medskip

\noindent
\textbf{Other initial distributions:} We may consider the case that initially each vertex is occupied by a single particle. Consider the case that $G$ is a $d$-regular Eulerian graph.  We outline in \S\ref{s:deter} the necessary adaptations required in order to show that also in this setup we have    \begin{equation}
 \label{e:berintro}
 \whp\footnote{We say that a sequence of events $A_n$ defined with respect to some probabilistic model on a sequence of graphs $G_n:=(V_n,E_n) $ with $|V_n| \to \infty $ holds $\mathrm{w.h.p.}$ (``with high probability")  if the probability of $A_n$ tends to 1 as $n \to \infty$. }  \quad \SC(G) \le  C d (\log |V|)^3. %
\end{equation}

\medskip

\noindent
\textbf{Extremality of the complete graph:}
We will prove Theorem \ref{thm: reg} for an arbitrary holding probability $p \in [0,1)$ for the walks (i.e.~the probability of a walker staying put at each given step is $p$). Let $K_n$ be the complete graph on $n$ vertices. Consider the SN model on $K_n$ with some holding probability   $p_n \in [ 0,1/n] $. It is not hard\footnote{The proof is similar to the proof that the connectivity threshold for $G(n,p)$, the Erd\H os and R\'enyi random graph, occurs in a window of order $1/n$ around $p_n=(\log n)/n$ (for the SN model the proof is still elementary, but it is somewhat tedious and thus omitted).} to verify that $\SC(K_n)$ is concentrated around $\log n$ in a window whose width is of order 1.  Thus, by \eqref{e:reglower1} the complete graph is the regular graph with (asymptotically) the fastest social connectivity time (at least when the holding probability is between 0 and $1/(d+1)$).

\medskip

\textbf{Extensions:}
In \S\ref{s: open} we discuss some  extensions of our results (e.g.~allowing the walkers to belong to different communities, walking on different graphs with the same vertex set).

\subsection{The $n$-cycle and higher dimensional tori}
\label{s:introcycle}
\begin{maintheorem}
\label{thm: cycle}
Let $\mathrm{C}_n$ be the $n$-cycle. Consider the setup in which at time 0 there is exactly one walker at each vertex and that the walkers perform continuous-time independent SRWs (with jump rate 1). Then there exist some absolute constants $c_1,c_2,\alpha,C_1,C_2>0$ such that
\begin{equation}
\label{eq: cycleupper}
\Pr[\SC(\mathrm{C}_n)> C_1 \log^2 n] \le C_{2}/n.
\end{equation}
\begin{equation}
\label{eq: cyclelower}
\Pr[\SC(\mathrm{C}_n)< c_1 \log^2 n] \le e^{-c_2n^{\alpha}}.
\end{equation}  
\end{maintheorem}
The proof of \eqref{eq: cycleupper} is due to Gady Kozma. We thank him for allowing us to present his argument. Working in continuous-time  simplifies
the analysis of \eqref{eq: cycleupper}. The proof in the usual setup of  $\mathrm{Pois}(1)$ walkers per site is in fact simpler (see Remark \ref{rem: Poistrick}). As mentioned after Theorem \ref{thm: reg}, in the usual Poisson, discrete-time setup \eqref{eq: cyclelower}  follows from \eqref{e:reglower2}.
\begin{remark}
\label{rem:lacycle}
When initially the number of particles per site is $\Pois(\la) $ distributed (independently for different sites) it is possible to extend the analysis of Theorem \ref{thm: cycle}  in order to show that $\whp $ $c< \sfrac{\SC(\mathrm{C}_n)}{\max \{ [ \la^{-1} \log (\la n)]^2,1\} } \le C $
as long as $\la \gg \sfrac{1}{n}$ (cf.\ \cite{BHM} for the analysis of the frog model on the $n$-cycle for $ \la \gg n^{-1}$).
\end{remark}
\begin{remark}
Denote the $d$-dimensional torus of side length $n$ by $\mathbb{T}_d(n) $. This is the Cayley graph of $(\Z/n \Z)^d$ obtained by connecting by an edge each $x,y \in (\Z/n \Z)^d $ which disagree only in one coordinate, by $\pm 1$ mod $n$. Then  $\mathbb{T}_1(n)$ is the $n$-cycle. We comment that using similar ideas as in \cite[Theorem 2]{BHM} one can show that $\whp$ $c< \sfrac{ \SC(\mathbb{T}_2(n))}{\log n \log \log n}<C $ and $c< \sfrac{ \SC(\mathbb{T}_d(n))}{d\log n }<C $ for $d \ge 3$. This result appeared in an earlier version of this work. Note that $\SC(\mathbb{T}_2(n)) \ge c \log n \log \log n $  and  $\SC(\mathbb{T}_d(n)) \ge (1-o(1))d\log n  $  $\whp$ follow from \eqref{e:reglower2}-\eqref{e:reglower1}.
\end{remark}

\subsection{Expanders}
We denote by $ \gamma(G)$  the \textbf{spectral-gap} of LSRW on $G$, defined as  the smallest non-zero eigenvalue of $I-P$, where $P$ is the transition matrix of LSRW  on $G$ and $I$ is the identity matrix.  
We say that a sequence of graphs $G_{n}$ is an expander family if $\inf_n \gamma(G_n)>0$. We say that $G$ is a $\gamma$-\textbf{expander}
if $\gamma(G)
\ge \gamma$ (we think of $\gamma$ as being uniformly bounded away from 0, independently of the size of $G$, although our analysis remains valid even if this fails). The following result demonstrates that Theorem \ref{thm: reg} is sharp. 
\begin{maintheorem}
\label{thm: et1}
Let $G$ be a connected $d$-regular $n$-vertex $\lambda$-expander. Then there exist absolute constants $C,C',C_0$ such that if   $ \lambda \log_{d} n \ge C_{0}  $, then $\SC(G) \le C \lambda^{-1} \log n$ w.p.~at least $1-C'/n$.
\end{maintheorem}
\begin{maintheorem}
\label{thm: et2}
Let $G$ be a connected $d$-regular $n$-vertex $\lambda$-expander. Then there exist some constants $c,C_{1}>0$ such that with probability at least $1-e^{-cnd^{-4(s+1)}}$, after $s:=\lceil C_{1}/\lambda \rceil$ steps there exists a class of walkers 
of size at least $n/6$.
\end{maintheorem} 
We refer to a class of walkers (in the SN model on some graph $G=(V,E)
$ at a certain time) of size $\Omega (|V|)$ as a \textbf{giant class}.
We shall deduce Theorem \ref{thm: et1} from Theorem \ref{thm: et2} by bounding $\SC$ by the time it takes the giant class from Theorem \ref{thm: et2} to ``capture" the rest of the walkers. The requirement    $ \lambda \log_{d} n \ge C_{0}  $ in  Theorem \ref{thm: et1} is imposed in order to ensure that the error term $e^{-cnd^{-4(s+1)}} $ from Theorem \ref{thm: et2} is small. 

\subsection{Organization of this work} 
In \S\ref{sec: preliminaries} we present some preliminaries. In \S\ref{s: upper} we prove \eqref{eq: upper}. In \S\ref{s: lower} we prove \eqref{eq: lower} and Theorem \ref{thm: reg}. In \S\ref{s: cycle} we prove Theorem \ref{thm: cycle}. In \S\ref{s: expanders} we study the SN model on expanders and prove Theorems \ref{thm: et1}-\ref{thm: et2}. We conclude with some open problems,  conjectures and concluding remarks in \S\ref{s: open}.

\subsection{Related work and connections to other models}
\label{s:related}
The social network model on infinite graphs is investigated in \cite{Hermon}. The setup studied in \cite{Hermon} is as follows. Given an infinite connected regular graph $G=(V,E)$,  place at each vertex Pois($\la$) walkers performing independent lazy simple random walks on $G$ simultaneously. When two walkers visit the same vertex at the same time they are declared to be
acquainted. It is shown in \cite{Hermon} that when $G$ is vertex-transitive and amenable, for all $\la>0$ $\as$ every pair of walkers will eventually have a path of acquaintances between them. In contrast, it is shown that when $G$ is non-amenable (not necessarily transitive) there is always a phase transition at some $\lac(G)>0$.  General bounds on $\lac(G)$ are given and the case that $G$ is the $d$-regular tree $\mathcal{T}_d $ is studied in more detail (it is shown that $c \le \lac(\mathcal{T}_d)/ \sqrt{d} \le C $). Finally, it is shown that in the non-amenable setup there exists a finite time $t$ such that $\as$ there exists an infinite set of walkers having a path of acquaintances between them by time $t$. We note that the last result is the infinite setup counterpart of our Theorem \ref{thm: et2}.

\medskip

We believe that, in spirit, the results in this paper should be true also for some other particle systems involving independent random walks. The  $A+B \to 2B$ family of models \cite{kesten2005spread,kesten2008shape}, often interpreted as models for a spread of an epidemic/rumor, are defined by the following rule: there are type $A$ and $B$ particles occupying a graph $G$, say with densities $\la_A,\la_B>0$. They perform independent SRW with holding probabilities $p_A \in [0,1]$ and $p_B \in [0,1)$. When a type $B$ particle collides with a type $A$ particle, the latter transforms into a type $B$ particle.
The \emph{frog model }is the particular case of  the above dynamics in which the type $A$ particles are immobile ($p_A=1$) and $\la_A=\la_B=\la $. 

\medskip

We now consider a variant of the above model which is intimately related to the SN model.  
Consider the case in which  initially only the particles at some vertex $o $ are of type $B$ and an additional  $B$ particle is planted at $o$ (this is done to ensure that there is at least one $B$ particle). Consider the case that each $B$ particle performs only $t$ steps of its walk before vanishing (after which it cannot become reactivated again). We call $t$ the lifetime of the $B$ particles. We say that the process dies out when there are no $B$ particles left.  One  then defines the \emph{susceptibility}, $\mathcal{S}(G) $ as the minimal lifetime which is sufficient so that all particles are transformed into type $B$ particles before the process dies out (to make this definition formal, we think of the particles as first picking the infinite trajectory of their random walks, but then, once transforming into a $B$ particle, performing only $t$ additional steps of their walk). In the interpretation of the model as a model for a spread of a rumor/epidemic the susceptibility is meant to capture ``how long should a virus live in order to infect the entire population" or "how interesting should a rumor be, so that eventually everybody will hear it". Similarly, one can consider the minimal lifetime required so that every vertex is visited by at least one $B$ particle.

In \cite{hermon2016frogs,BHM} parallel results to Theorems \ref{thm: reg},\ref{thm: cycle},\ref{thm: et1} and \ref{thm: et2}  are proved for the above variant in the case of the frog model, where the susceptibility replaces the social connectivity time. We strongly believe that (certain variants of) Theorems \ref{thm: reg},\ref{thm: cycle},\ref{thm: et1} and \ref{thm: et2} should hold even when the $A$ particles are mobile. Moreover, we have the following conjecture concerning an analog of Theorem \ref{thm: 1}.

\begin{conjecture}\label{con:polylog} Consider the aforementioned variant of the $A+B \to 2B$ model with $\la_A=\la_B=\la$ and some fixed  $p_A \in [0,1]$ and $p_B \in (0,1)$.  Then there exist some $C_{d,\la},\ell>0$, such that for every sequence of finite connected graphs $G_{n}=(V_{n},E_{n})$ with $|V_n| \to \infty$ of maximal degree at most $d$, we have that \[\lim_{n \to \infty} \mathbb{P}_{\la}[\mathcal{S}(G_n) \le C_{d,\la}  \log^{\ell} |V_n| ]=1.\]
\end{conjecture}  
   
\section{Preliminaries and additional notation}
\label{sec: preliminaries}
\subsection{Reversibility,  stationarity of the occupation measure and independence of the number of walkers performing different walks.}
Let $G=(V,E)$ be an Eulerian digraph. Let $P$ be the transition kernel of LSRW on $G$. We denote by $P^t(u,v)$ the $t$-steps transition probability from $u$ to $v$. We now establish a certain independence property for walks in $G$, which in particular implies a certain stationarity property of the SN model.  Recall that $\pi_v:=d_v/2|E|$ and $\bar \pi=|V|\pi $. For LSRW on a graph $G$ \textbf{reversibility} is the property that for every $v,u \in V$ and $t \ge 0$, \[\pi_vP^t(v,u)=\pi_uP^t(u,v).\]  A \textbf{walk} of length $k$ in $G$ is a sequence of $k+1$ vertices $(v_0,v_1,\ldots,v_{k})$ such that for all $0 \le i <k$ either $v_i=v_{i+1}$ or $(v_{i} , v_{i+1}) \in E $ (or for a graph $\{v_{i} , v_{i+1}\} \in E $). Let $\Gamma_k$ be the collection of all walks of length $k$ in $G$.
 For a walk $\gamma=(\gamma_0,\ldots,\gamma_k) \in \Gamma_{k}$ for some $k \ge 1$, let \[p(\gamma):=\prod_{i=0}^{k-1}P(\gamma_i,\gamma_{i+1}) \quad \text{and} \quad q(\gamma):= \bar \pi_{\gamma_0}p(\gamma).\] Note that $p(\gamma)$ is precisely the probability that a  walker whose starting position is $\gamma_0$  performed the walk $\gamma$ (i.e.~her position at time $i$ is $\gamma_i$ for all $i \le k$). Let $\gamma' $ be the reversal of $\gamma $ (i.e.~$\gamma'_{i}=\gamma_{|\gamma|-i} $ for all $ i \le |\gamma| $). Reversibility is equivalent to the following condition:
$$q(\gamma)=q(\gamma'), \quad \text{for every walk }\gamma . $$ 
We denote the number of walkers who performed a walk $\gamma$ by $X_{\gamma}$. We denote
the number of walkers whose position at time $t$ is $v$ by $Y_{v}(t)$. 
Since $\pi$ is the stationary distribution of the walks, 
$\mathbb{E}[Y_v(t)]=  \sum_{u
\in V} \bar \pi_u P^t(u,v) = \bar \pi_{v}=\mathbb{E}[Y_v(0)]$,  for every vertex $v$ and time $t>0$. Thus,
the following fact follows easily from Poisson thinning.   
\begin{fact}
\label{fact: thinning}
Let $G$ be an Eulerian digraph. Then for all $t>0$ and $\gamma \in \Gamma_t$, $X_{\gamma} \sim \Pois(q(\gamma)) $.  Moreover, $(X_{\gamma})_{\gamma \in \Gamma_t} $ are independent for each fixed $t>0$. Consequently, for each fixed $t>0$ we have that $(Y_v(t))_{v \in V} $ are independent and that $Y_{v}(t) \sim \mathrm{Pois}(\bar \pi_v) $ for all $v \in V $.
\end{fact}

Next we note that if the average degree of $G=(V,E)$ is $d$, then for all $u ,v\in V$
\begin{equation}
\label{eq: ratiod}
 \bar \pi_v=\sfrac{|V|d_v}{|E|} =d_v/d , \quad  \bar \pi_v / \bar \pi_u =d_{v}/d_{u}  .
\end{equation}
We shall use  the following concentration estimate.  Let $Y \sim \mathrm{Pois}(\mu)$. Then 
\begin{equation}
\label{eq: poisconcentration}
\forall a \ge 0, \quad \max \{ \Pr[Y \ge \mu + a], \Pr[Y \le \mu - a]  \} \le \exp (-\sfrac{a^2}{2(\mu+a/3)} ).
\end{equation}
By taking an appropriate limit, this follows  from Bernstein's inequality which says that if $\xi_1,\ldots,\xi_m$ are independent Bernoulli r.v.'s, $Y=\sum_{i \in [m]}\xi_i $, $\mu=\mathbb{E}[Y]$ and $\sigma^2= \mathrm{Var}(Y) $, then
\begin{equation}
\label{e:berns}
\forall a \ge 0, \quad \max \{ \Pr[Y \ge \mu + a], \Pr[Y \le \mu - a]  \} \le \exp (-\sfrac{a^2}{2(\sigma^2+a/3)} ).
\end{equation}

Another concentration estimate which we shall use in the proof of some of the lower bounds on $\SC$ is the McDiarmid''s inequality, which we now state. Let $f: \mathcal{X}^n \to \R $. Let $c_i:=\sup|f(\mathbf{x} )-f(\mathbf{y} )| $, where the supremum is taken over all $\mathbf{x}:=(x_1,\ldots,x_n),\mathbf{y}:=(y_1,\ldots,y_n) \in \mathcal{X}^n $ such that $x_j=y_j $ for all $j \neq i $.  Let $X_1,\ldots,X_n $ be i.i.d.\ $\mathcal{X}$ valued r.v.'s. Then
\begin{equation}
\label{e:Mcstatement}
\Pr[f(X_1,\ldots,X_n)< \mathbb{E}[f(X_1,\ldots,X_n)] - \eps ] \le \exp (-\sfrac{2\eps^2}{\sum_{i \in [n]}c_i^2} ). \end{equation} 
\subsection{Negative association}
Let $G=(V,E)$ be a finite Eulerian digraph. The goal of this subsection is to formulate the intuitive notion that given a certain deterministic initial configuration of walkers, a time $t>0$ and a vertex $x$, the more walkers there are at $x$ at time $t$, the less walkers we expect to find at the other vertices of $G$. 

 The partial order $\preceq$ on $\Z_{+}^V $ is defined as follows. For $a,b \in \Z_{+}^V $ we have that $a \preceq b  $ if $a(v) \le b(v)$ for all $v$. We say that a function $f : \Z_{+}^V \to \R   $ is \emph{increasing} (resp.\ \emph{decreasing}) if $f(a) \le f(b) $ whenever  $a \preceq b $ (resp.\ whenever  $b \preceq a $, equivalently, $-f$ is increasing).   

 Let $\mathbf{X}_1:=(X_{1}(v))_{v \in V }$ and $\mathbf{X}_2:=(X_{2}(v))_{v \in V } $ be random elements in $\Z_+^V $. Denote their laws by $\mu_1$ and $\mu_2$, respectively.  We say that (the law of) $\mathbf{X}_1$ \emph{stochastically dominates} (the law of) $\mathbf{X}_2$ if  $\int f d \mu_2 \le \int f d \mu_1 $ for all increasing $f$. We say that $f:\Z_{+}^V \to \R $ depends only on the co-ordinates in $A$ if $f(z)=f(z') $ whenever $z(a)=z'(a) $ for all $a \in A$. We say that $\mathbf{Y}:=(Y(v))_{v \in V } $ are \emph{negatively associated} if for every two disjoint sets $A,B \subset V $ and any two increasing functions $f,g $ such that the former depends only on the co-ordinates in $A$ and the latter
depends only on the co-ordinates in $B$ we have that $\mathbb{E}[f(\mathbf{Y})g(\mathbf{Y})] \le \mathbb{E}[f(\mathbf{Y})] \mathbb{E}[g(\mathbf{Y})]$. It is immediate from the definition that if  $(Y(v))_{v \in V } $ are negatively associated then: \begin{itemize}
\item[(1)] The law of $(Y_u)_{u:u \neq v } $ stochastically dominates its conditional law given $Y_v >0 $ (take above $A=\{v\},B=V \setminus\{v\} $ and $f=\mathbf{1}_{\{Y_v \ge 1\} }$). 
\item[(2)]$(Y(v))_{v \in V }$ are negatively correlated, i.e.\  $\mathbb{E}[Y_{v}Y_{u}] \le \mathbb{E}[Y_{v}] \mathbb{E}[Y_{u}]$ for all $u \neq v$.
\item[(3)] If $V=[n]$ then for all increasing (resp.\ decreasing) $f_1,\ldots f_n: \Z_+ \to \R $ we have that
\begin{equation}
\label{e:negaff} \mathbb{E}[\prod_{i \in [n]}f_i(Y_i) ] \le  \prod_{i \in [n]}\mathbb{E}[f_i(Y_i) ]. \end{equation} 

\end{itemize}
The following result is due to Dubhashi and Ranjan \cite[Theorem 13]{NA}.
 \begin{fact}
\label{cor:dommultinomial}
Let $X_1,\ldots,X_m$ be independent $[n]$ valued random variables. For $i \in [n]$ let $Y_i:=\sum_{j \in [m] } \mathbf{1}_{\{X_{j} =i\} }  $ be the number of $X_{\ell}$'s which take the value $i$. Then $(Y_i)_{i \in [n] }$ are negatively associated.
\end{fact}

\section{A general upper bound on $\SC$ - proof of \eqref{eq: upper}}
\label{s: upper}
In this section we prove \eqref{eq: upper}, which is the harder and more interesting part of our main result, Theo

rem \ref{thm: 1}. Throughout the section we fix some connected $n$-vertex Eulerian digraph $G=(V,E)$ of average degree $d$ and  set \begin{equation}
\label{e:t}
t:=t_{d,|V|}:=  \lceil Cd^{1_{G \text{ is not regular} }}   \log  n \rceil^{2} 
\end{equation}
for some sufficiently large absolute constant $C$ to be determined later.
\subsection{An overview of the proof}
We first describe an easy auxiliary baby model (\S\ref{s:baby}) and analyze it (Lemma \ref{lem: mainlem}). We then (\S\ref{s:redbaby}) define the notions of a marginally $(\alpha,t)$-merging configuration, a globally  $(\alpha,p,t)$-merging configuration and a $(\delta,t)$-balanced configuration (Definition \ref{def:ballanced}) and explain how the analysis of the SN model can be reduced to that of the auxiliary baby model, assuming one of the first two properties hold at each given time $\whp$ (Proposition \ref{p:derof1.1}). Finally, in \S\ref{s:checkingbalance} we show that at each given time  $\whp$ the configuration is $(1/8,t)$-balanced, and that this implies that it is marginally $(\alpha,t)$-merging for some $\alpha=\alpha_d$.  

\subsection{Auxiliary baby models}
\label{s:baby}
Consider an arbitrary discrete-time coalescent process starting with $m$ distinct classes. Define the coalescence time $\mathrm{CT}$ as the minimal time in which there is only one class (if there is no such time it is defined to be $\infty$). We now describe several simple coalescence models, and derive upper bounds on the coalescence time in each model.  We intentionally start with two simple models, each involving a deterministic condition, and then work our way towards a slightly more complicated model which is more faithful to the actual details of our argument for the SN model. 

\medskip

Clearly,  if prior to $\mathrm{CT}$ in every time unit every class merges with at least one other class, then deterministically, it must be the case that $\mathrm{CT} \le  \lceil \log_2 m \rceil $. We now generalize this simple model, while introducing some terminology which will be used throughout the section.
\begin{definition}
\label{def:round}
Fix some $t>0$. Let $t_k:=kt$. We call the time interval $(t_k,t_{k+1}]$ the $(k+1)$\textbf{\emph{-th round}}. We say that the $(k+1)$-th round is a $p$-\textbf{\emph{merging round}} if either $t_k \ge \mathrm{CT}$ (this case is taken for technical reasons) or  $t_k< \mathrm{CT}$ and at least a $p$-fraction of the classes (w.r.t.~time $t_k$) have merged with at least one other class by the end of time $t_{k+1}$.
\end{definition}
The last definition makes sense for the SN model, as well as for a general coalescence model.
Note that if the number of classes at the beginning of a $p$-merging round is $\ell>1$, then the number of classes at the end of the round is at most $(1-p/2)\ell$. Assume that initially there are at most $m$ classes. If all rounds are $p$-merging then,  deterministically, $\mathrm{CT}\le \lceil C_p  t \log m \rceil $ (for concreteness, $C_p:=1/\log (\frac{2}{2-p})$). More generally, even if not all rounds are $p$-merging, there can be at most $\lceil C_p \log m \rceil $ $p$-merging rounds prior to $\mathrm{CT}$. This motivates the following simple model in which each round is $p$-merging with probability at least $\alpha$, regardless of the history. We intentionally keep the description and analysis of this model somewhat informal, deferring a more rigorous statement to Lemma \ref{lem: mainlem}.

\medskip

\textbf{Simplified baby model:} Start with $m$ classes. At each round flip a  coin with heads probability at least $\alpha$ (i.e.~the heads probability in each round may depend on the history up to that time, but is guaranteed to be at least $\alpha$). If it lands on heads then the round is $p$-merging. The list of pairs of classes which merge with one another in each round (given the evolution of the model up to the beginning of the round and the result of the coin flip corresponding to the round) is determined according to some arbitrary predetermined rule (satisfying that whenever the coin lands on heads the requirement of being $p$-merging is satisfied by the list of pairs of merged classes), perhaps using some additional independent source of randomness. The predetermined rule can be thought of as one decided by  an adversary.

\medskip

It is intuitively clear that there exists $C_{\alpha,p}$ of the form $C_{\alpha,p} = K_{1}C_p/\alpha \le  \frac{K_{2}}{\alpha p}$ (for some sufficiently large absolute constants $K_{1},K_2$, where $C_p$ is as above) such that the probability that the total number of rounds exceeds $L=L_{m,\alpha,p}:=\lceil C_{\alpha,p}\log m \rceil$ tends to 0 as $m \to \infty$. Informally, one can argue that the number of $p$-merging rounds  in the first $L$ rounds, stochastically dominates the $ \mathrm{Bin}(L,\alpha)$ distribution and then  apply an appropriate large deviation estimate in order to bound $\Pr[ \mathrm{Bin}(L,\alpha) < C_p \log m] $.

\medskip
 
 \textbf{The auxiliary baby model:} We now consider another small variation of the last model by introducing an error term $\epsilon_m$. We refer to this variation as the \emph{auxiliary baby model}. Consider the case that at each round with probability we only have probability $ \ge  1-\epsilon_m$, for some $\epsilon_m=o(1/\log m)$, that the model evolves according to the above description (i.e.~like the ``simplified baby model"), while with the complement probability it stops evolving (i.e.\ the coalescence process is terminated). It is intuitively clear that the conclusion from the previous paragraph remains valid, with an additional error term of $\eps_m L=o(1)$ for the probability that the number of rounds exceeds $ L$. We will show that the analysis of the SN model on an Eulerian graph $G=(V,E)$ of average degree $d$ can be reduced to that of the auxiliary baby model (for a suitable choice of parameters). 

\medskip
 
  Observe that above it was sufficient to consider the number of classes at the end of each round in order to bound $\mathrm{CT} $ from above.  This motivates the following lemma, which formulates rigorously the assertions made in the previous paragraphs. One should think of $J_k+1$ below as the number of classes at the end of the $k$-th round, and so $T_{0}:=\inf \{k: J_k=0 \}$ corresponds to the coalescence time $\mathrm{CT}$. 
\begin{lemma}
\label{lem: mainlem}
Let $J_k$ be a non-increasing $\Z_+$-valued random process, measurable w.r.t.~filtration $(\F_k)_{k \ge 0}$, with $J_0 \le m$, satisfying that for some $0<\alpha,p,\eps_m<1$, for each $k$ there exists some events $A_k \in \F_k$ with $\Pr(A_k) \ge 1-\eps_m$ satisfying that
\begin{equation}
\label{eq: conditionalexp}
 \mathbb{E}[1_{\{J_{k+1} \le (1-p/2) J_k \}}1_{A_k}  \mid \F_k ] \ge \alpha1_{A_k}. 
\end{equation}
Let $T_0:=\inf \{k:J_k=0 \}$. Then there exists some constant $C_{\alpha,p} \le K/(\alpha p)$ such that 
$$\Pr[T_0 > \lceil C_{\alpha,p}\log m \rceil] \le m^{-2}+\eps_m (\lceil C_{\alpha,p}\log m \rceil +1). $$ 
\end{lemma}
The proof of Lemma \ref{lem: mainlem} follows by noting that by \eqref{eq: conditionalexp} $M_k:=(1-\frac{1}{2} \alpha p )^{-k}J_{k}1_{\cap_{i \in [k]}A_i}$ is a supermartingale and $\Pr[T_0>k]=\mathbb{E}[J_{k}1_{\cap_{i \in [k]}A_i}]+\Pr[\cup_{i \in [k]}A_i^c] \le (1-\frac{1}{2} \alpha p )^{k}\mathbb{E}[M_0]+\eps_m k $. The details are given in the appendix in \S\ref{s:proofl4.2}.

\subsection{Reducing the SN model to the baby model via merging configurations}
\label{s:redbaby}
We are interested in reducing the SN model to the auxiliary baby model.  Recall that $Y_v(s)$ is the number of walkers at vertex $v$ at time $s$. It is beneficial to consider a notion of ``independence of the history" which depends only on $(Y_v(s))_{v \in V}$, whose distribution is stationary and hence convenient to work with, rather than also on the identity of the classes of walkers at time $s$, which might have a complex dependency structure.   
This motivates the following key definitions.
\begin{definition}
\label{def:ballanced}
\begin{itemize}
Let $\mathbf{Y}:=(Y_v)_{v \in V} \in \Z_{+}^V$. We identify $\mathbf{Y}$ with a configuration of walkers in which for each $v \in V$ there are initially $Y_v$ particles at $v$. The walkers then perform a LSRW on $G$ of length $t$, where $t$ is as in \eqref{e:t}.
 \item
Let $\mathcal{P}$ be a partition of the walkers into (disjoint) sets  $\mathcal{A}_1,\ldots,\mathcal{A}_m $ with $m \ge 2$.
We say that $\mathcal{P}$  \textbf{\emph{respects vertices}} if walkers which initially occupy the same vertex all belong to the same set in the partition (i.e.~for all $v$, each $\mathcal{A}_i$ contains either all $Y_v$ walkers initially occupying $v$ or none of them).
\item
We say that the configuration $\mathbf{Y}$ is \textbf{\emph{marginally}} $(\alpha,t)$-\textbf{\emph{merging}}, if the following holds. For every vertex-respecting partition of the walkers into sets, $\mathcal{A}_1,\ldots,\mathcal{A}_m $ with $m \ge 2$,  for every $\mathcal{A}_i$ which contains less than half of the walkers, the probability that at least one walker from $\mathcal{A}_i$  will meet some walker not belonging to $\mathcal{A}_i$ by time $t$ is at least $\alpha$. We  say such a configuration is \emph{\textbf{globally}} $(\alpha,p,t)$-\emph{\textbf{merging}} if for every vertex-respecting partition of the walkers into sets, $\mathcal{A}_1,\ldots,\mathcal{A}_m $ with $m \ge 2$, we have that with probability at least $p $, there are at least $\alpha (m-1) $ $i$'s such that at least one walker from $\mathcal{A}_i$   met by time $t$ some walker not belonging to $\mathcal{A}_i$.   

 \item We say that the configuration $\mathbf{Y}$ is $(\delta,t)$-\textbf{\emph{balanced}}, if for all $v \in V$ we have that
\begin{equation}
\label{eq: defbalanced}
\sum_{u \in V}Y_uP^t(u,v) \ge (1-\delta)\bar \pi_v \quad \text{(recall that $\bar \pi_v:=|V|\pi_v$)},
\end{equation}
and
\begin{equation}
\label{eq: defbalanced2}
\max_{u \in V} Y_u /d_u \le \log |V|.
\end{equation}
\item We say that the configuration $\mathbf{Y}$ is \textbf{fully} $(\delta,t)$-\textbf{\emph{balanced}},
if it is  $(\delta,t)$-balanced and in addition $\sum_{u \in V}Y_uP^t(u,v) < (1+2\delta)\bar \pi_{v} $ for all $v \in V$.
\end{itemize}
\end{definition} 
We believe that some of the ideas from the proof of \eqref{eq: upper} and in particular the notion of balanced configurations can be useful for other particle systems. For instance, the approach taken in \cite{exclusion}, where the mixing time of the exclusion process on regular graphs is bounded from above in terms of the mixing of independent random walks, was greatly inspired by the proof of \eqref{eq: upper}.

Throughout, we take $\delta $ to be some constant in $(0,1/8]$. We now record the fact that requirement \eqref{eq: defbalanced2} holds $\whp$  when $(Y_u)_{u \in V} $ are independent and  $Y_u \sim \Pois(\bar \pi _u)$.
\begin{lemma}
\label{lem:maxPois}
Let $(Y_u)_{u \in V} $ be independent and  $Y_u \sim \Pois(\bar \pi _u)$. Then
\[\Pr[\max_{u \in V} Y_u /d_u > \log |V| ] \le \sum_{u \in V }C \left(\frac{e}{\log |V|} \right)^{d_u \log |V| } . \]
\end{lemma}   
\noindent
\emph{Proof:} $\Pr[\mathrm{Pois}(\bar \pi_{u} ) > d_u \log |V| ] \le \Pr[\mathrm{Pois}(d_u ) > d_u \log |V| ] \le C \Pr[\mathrm{Pois}(d_u ) = \lceil d_u \log |V| \rceil ]$. Finally, apply a union bound over all $u \in V$. \qed

\medskip

Consider the SN model viewed only at integer multiples of $t$. Think of the time interval $(kt,(k+1)t]$ as the $k$-th round (for $k=0$ we take $[0,t]$ instead of $(0,t]$)
Assume that the configuration at time $kt$ is globally $(\alpha,p,t)$-merging. Then the $k$-th round is $\alpha/2$ merging in the sense of Definition \ref{def:round} w.p.\ at least $p$. In light of  Lemma \ref{lem: mainlem}, in order to prove \eqref{eq: upper}, we need to show that for some absolute constant $c_1$, at each given time the configuration is globally $(c_{1}/d^{1+1_{G \text{ is not regular}}},c_1,t)$-merging, with probability at least $1-|V|^{-2}$, where here $t$ is as in \eqref{e:t}. This observation is recorded in Proposition \ref{p:derof1.1} below.  

A priori, it may appear difficult to determine that a configuration is either marginally $(\alpha,t)$-merging or globally $(\alpha,p,t)$-merging (for some $\alpha,p \in (0,1)$), as the sets  $\mathcal{A}_1,\ldots,\mathcal{A}_m $ from Definition \ref{def:ballanced} form an arbitrary vertex-respecting partition. On the other hand, the property of being $(\delta,t)$-balanced is more tractable as it is stated in terms of a (lack of) a large deviation behavior for the expected number of walkers occupying each site $t$ steps into the future, given the current configuration.
 We now explain the relations between the four properties of being $(\delta,t)$-balanced, fully  $(\delta,t)$-balanced, marginally $(\alpha,t)$-merging and globally $(\alpha,p,t)$-merging.    

  To prove \eqref{eq: upper}  we shall later show that (i) if $C$ from the definition of $t$ in \eqref{e:t} is taken to be sufficiently large then at each given time the configuration of walkers is fully $(\sfrac{1}{8},t)$-balanced w.p.\ at least $1-2|V|^{-2} + C'|V|^{2-\log \log |V|} $ and (ii) if $G$ is regular then a fully $(\sfrac{1}{8},t)$-balanced configuration is globally $(c_{1}/d,c_1,t)$-merging for some absolute constant $c_1>0$. 

Before proving (ii) we show that (iii) a $(\sfrac{1}{8},t)$-balanced configuration  is marginally \\ $(c_{1}/d^{1+1_{G \text{ is not regular}}},t)$-merging for some absolute constant $c_1>0$ and (iv) a  marginally  $(\alpha,t)$-merging configuration is globally $(\alpha/2,\alpha/4,t)$-merging (this is done in Corollary \ref{cor:glob}).   

\medskip

To prove (iv)  we use the following simple  lemma which asserts that a uniform lower bound on the marginal probabilities of certain $m$ events, implies a corresponding lower bound on the probability that a certain fraction of them occur. The first equation from the lemma follows by a straightforward application of the Paley-–Zygmund
inequality (e.g.~\cite[Lemma 4.1]{kallenberg2002foundations}). The second equation is the one-sided Chebyshev's inequality. 

\begin{lemma}
\label{lem: 2ndmoment}
Let $\xi_1,\ldots, \xi_m$ be indicator random variables. Let $p_i:=\E[\xi_i]$,  $p:=\min_i p_i$,  $S:=\sum_{i=1}^m \xi_i $,  $\mu=\mathbb{E}[S]$ and $\sigma^2:=\mathrm{Var}(S)=\sum_{i}p_i(1-p_i) +2\sum_{i,j:i<j}\mathrm{Cov}(\xi_i,\xi_j) $. Then
\begin{equation}
\label{eq: 2ndm}
\Pr[S \ge pm/2] \ge \Pr[S \ge \mu /2] \ge (\mu /2)^2 /\mathbb{E}[S^{2}] \ge \mu/(4m) \ge p/4.
\end{equation}
\begin{equation}
\label{eq: onesidedChebyshev}
\forall c \ge 0,\quad \Pr[S \ge \mu -c \sigma] \ge \sfrac{c^2}{1+c^2}.
\end{equation}
In particular, if $\mathrm{Cov}(\xi_i,\xi_j) \le 0 $ for all $i \neq j$ then $ \sigma^2 \le \mu$ and so $\Pr[S < \half \mu ] \le \sfrac{4}{4+\mu } $. 
\end{lemma}

In our setup, \eqref{eq: 2ndm} implies the following. Assuming that the configuration at time $t_{k}=kt$ is marginally $\alpha$-merging  and $\mathcal{A}_1,\ldots,\mathcal{A}_m $ are the classes of walkers at time $t_k$ (the classes clearly form a vertex-respecting partition of the walkers), then with probability at least $\alpha/4$, there exists an index set $I$ of size at least $ \lceil \alpha (m-1)/2 \rceil$ such that for all $i \in I$ there is some walker from $\mathcal{A}_i$ who met in the time interval $(t_k,t_k+t]$ some walker not belonging to $\mathcal{A}_i$.  We now record this observation.
\begin{corollary}
\label{cor:glob}
A marginally   $(\alpha,t)$-merging configuration is globally $(\alpha/2,\alpha/4,t)$-merging. 
\end{corollary}   
%
%
%
%
\textbf{Derivation of \eqref{eq: upper} assuming that at each given time the configuration is marginally or globally merging  $\whp$:}
\begin{proposition}
\label{p:derof1.1}
Let $G=(V,E)$ be a connected Eulerian digraph. Fix some $t>0$ and $\alpha \in (0,1) $. Assume that for each time $s$ w.p.\ at least $1-\eps_n$ the configuration of walkers at  time $s$ is  globally $(\alpha,p,t)$-merging. Then there exists an absolute constant $C>0$ such that
\[\Pr[\SC(G) \ge C (\alpha p )^{-1} t \log |V| ] \le |V|^{-2}+\eps_n C (\alpha p)^{-1}\log |V|.   \]
In particular, if at
each given time w.p.~at least $1-\eps_n$ the configuration of walkers is marginally $(\alpha,t)$-merging    then
\[\Pr[\SC(G) \ge C' \alpha  ^{-2} t \log |V| ] \le |V|^{-2}+\eps_n C' \alpha ^{-2}\log |V|.   \]
\end{proposition}
\begin{proof}
The second assertion follows from the first via Corollary \ref{cor:glob}.  
For the first assertion, if we think of each interval of length $t$ as a round, we have that if at time $kt$ the configuration of walkers is globally $(\alpha,p,t)$-merging, then the $(k+1)$-th round (corresponding to the time interval $(kt,(k+1)t]$) is $\frac{\alpha}{4}$-merging, in the sense of Definition \ref{def:round}, with probability at least $p$ (this holds regardless of the outcomes of the previous rounds). Since at the beginning of each round w.p.~at least $1-\eps_n$ the configuration of walkers is   globally $(\alpha,p,t)$-merging,  we are precisely at the setup of the auxiliary baby model and the assertion of the proposition now follows from Lemma \ref{lem: mainlem}. \end{proof}

\subsection{Proof of the validity of the marginally $(\alpha,t)$-merging property via a reduction to the $(\delta,t)$-balance condition}
\label{s:checkingbalance}

Note that if $(Y_v(s))_{v \in V}$ is distributed like the configuration of walkers in the SN model on $G$ at time $s$, then  (by stationarity of the occupation measure, which allows one to consider only the case that $s=0$) for every $v \in V$ the sum $\bar Z_v(s):= \sum_u P^t(u,v)Y_{u}(s)$  in \eqref{eq: defbalanced} (whose mean is $\bar \pi_v $) is a linear combination (with non-negative coefficients) of independent Poisson r.v.'s.  As demonstrated in Lemma \ref{lem: Poissum} below, in order to obtain a large deviation estimate for such a linear combination (namely, for the event $\bar Z_v(s) < (1-\delta)\bar \pi_v=(1-\delta)\mathbb{E}[ \bar Z_v(s)]  $), it suffices to control the maximal coefficient $P^t(u,v)$ appearing in the sum. 
We obtain such control using the following decay estimate for $\max_{x,y} P^{t}(x,y)$.

Recall that $d_y$ is the degree of vertex $y$.  
\begin{fact}[\cite{PS} Lemma 2.4]
\label{fact: decay}
There exists an absolute constant $M$ such that for every finite connected Eulerian digraph\    $G=(V,E)$,     LSRW on $G$ satisfies
\begin{equation}
\label{eq: returnp}
\forall t \ge 1, \quad \max_{x,y} |P^{t}(x,y)-\pi_y | \le M   \sqrt{1/t} [d_{y}1_{G \text{ is not regular}}+1_{G \text{ is regular}} ].
\end{equation}
\end{fact}
In what comes, we shall assume that  (*) $12d^{1_{G \text{ is not regular} }} \log |V|\le \delta^2|E|$.  Note that if (*) fails, then $|V|$ has to be bounded (recall that $\delta$ is a constant and that $|V| d = 2|E|$, as $d$ is the average degree). The following proposition, whose proof is deferred to \S\ref{s: larged} implies \eqref{eq: upper} when (*) fails. 
\begin{proposition}
\label{prop: larged}
Let $G=(V,E)$ be a connected Eulerian digraph. Then for some $C>0 $,
\[ \Pr[\SC(G)>C|E||V| \log|V|] \le C/|V|. \]
\end{proposition}
By Fact \ref{fact: decay}, if we set $C=12M \delta^{-2}$  in the definition of $t$ (\eqref{e:t}) then \begin{equation}
\label{eq: px,y2}
\max_{x,y}P^{t}(x,y) \le \pi_y+ \sfrac{\delta^{2}d_{y}}{12d\log |V|}.
\end{equation}

 Assume (*) holds for some fixed $0<\delta \le 1/8$. By \eqref{eq: px,y2} we have that
\begin{equation}
\label{eq: px,y3}
\max_{x,y}P^{t}(x,y) \le  \delta^{2}d_{y}/(6d \log |V|).
\end{equation}
In light of Lemma \ref{lem:maxPois} and the paragraph following \eqref{eq: defbalanced},  by combining \eqref{eq: px,y3} with the following lemma,   it follows by a union bound over the vertices of $G$ that the probability that the configuration is not $(\delta,t)$-balanced at some fixed time $s$ is at most $C_{1}/|V|^{2}$ (assuming (*), when the constants are chosen as indicated above).

\begin{lemma}
\label{lem: Poissum}
Let $\xi_1,\ldots, \xi_m$ be independent Poisson random variables. Let $p_1,\ldots,p_m \in (0,1)$. Let    $p_*:=\max p_i$,  $S:=\sum_{i=1}^m p_{i} \xi_i $ and $\mu:=\mathbb{E}[S]$. Then for all $\eps \in (0,1]$
\begin{equation}
\label{eq: poissonconcentration}
\begin{split}
& \Pr[S \le (1 - \eps)\mu ] \le e^{-\half \mu \eps^2/p_* }.
\\ & \Pr[S \ge (1+\eps)\mu ] \le e^{-\sfrac{1}{4} \mu \eps^2/p_* }.
\end{split}
\end{equation} 
\end{lemma}
The proof of Lemma \ref{lem: Poissum} is given in the appendix in \S\ref{s:proofl4.8}.
\begin{corollary}
\label{cor:LDforfuture}
Let $Y_v(s)$ be the number of walkers at vertex $v$ at time $s$. Assume that  $Y_{v}(0)\sim \mathrm{Pois}( \bar \pi_v )$ independently for different vertices. Let $t $ be as above  and assume that  (*) holds. Let $n:=|V|$ and $\bar Z_v(s):= \sum_{u \in V } P^t(u,v)Y_{u}(s) $.  Then for all $s \ge 0 $ and $v \in V $ we have
\[\Pr[ \bar Z_v(s)\le (1-\delta)\bar \pi_v ] \le \exp [-\sfrac{\bar \pi_v \delta^2}{2 (\delta^{2}d_{v}/(6d \log n))}  ] \le n^{-3}. \]   
\[ \Pr[\bar Z_v(s)\ge (1+2\delta)\bar \pi_v ] \le \exp [-\sfrac{\bar \pi_v (2\delta)^2}{4(\delta^{2}d_{v}/(6d \log n))}] \le n^{-6}.\]
\end{corollary}

The following proposition concludes the proof of  \eqref{eq: upper} when $G$ is not regular, and provides a  version of \eqref{eq: upper} also for a regular $G$ which has a worse degree dependence than \eqref{eq: upper}. Below we assume that the total number of walkers is at most $|V|+|V|^{2/3}$.   By \eqref{eq: poisconcentration}, the probability that this fails  is at most $\exp(-\frac{1}{3}|V|^{1/3})$ and hence  there is no harm in assuming this holds.
\begin{proposition}    
\label{prop: main}
Let $G=(V,E)$ be an $n$-vertex connected Eulerian digraph of average degree $d$. Let $\xi:=\Ind{G \text{ is not regular}} $. Let $0<\delta \le 1/8$. Assume that (*) holds.  Assume that a configuration of walkers is $(\delta,t)$-balanced for
$ t:=\lceil12d^{\xi} M \delta^{-2} (\log  n) \rceil^{2} $
 and that the   total number of walkers it contains is smaller than $n+n^{2/3}$.  Then there exists some absolute constant $\bar c >0 $ such that the following hold.
\begin{itemize}
\item[(i)]
The configuration is marginally $(\bar c/d^{1+\xi},t)$-merging.
\item[(ii)] If  $G$ is regular and $|\{(u,v) \in E : u \in A,v \notin A  \text{ or }u \notin A,v \in A \}| \ge c_0 d $ for every set $A \subset V$ such that $|A| \le \sfrac{2}{3} (n+n^{2/3}) $ then the configuration is marginally $(\bar c c_{0},t)$-merging. \end{itemize}
\end{proposition}
\begin{proof}
 Consider some configuration $(X_v)_{v \in V}$ of walkers (i.e.~for all $v$ the number of walkers at vertex $v$ is $X_v$) with at most $n+n^{2/3}$ walkers (i.e.\ $\sum_v X_v \le n+n^{2/3}$) and a vertex-respecting partition of the walkers into disjoint sets   $\mathcal{A}_1,\ldots,\mathcal{A}_m$. Since both the properties of a configuration which we consider (being marginally $(\alpha,t)$-merging and $(\delta,t)$-balanced) are independent of the time, by stationarity of the occupation measure w.l.o.g.~we may assume that this configuration corresponds to time 0.
We start with part (i). 

Let $\alpha_d:=\bar c/d^{1+\xi}$, where $\bar c>0$  shall be determined soon. We now fix some $1 \le i \le m$ such that $|\mathcal{A}_i| \le (n+n^{2/3})/2 $ and show that the probability that at least one walker from $\mathcal{A}_i$ meets some walker not from $\mathcal{A}_i$  by time $t$ is at least $\alpha_d$.  Denote the collection of vertices occupied by the walkers from $\mathcal{A}_j$ (at time 0) by $V_j$ (where $1 \le j \le m$). Denote the expected number of walkers from $\mathcal{A}_j$ (resp.\ not from $\mathcal{A}_j $, i.e.\ from $\cup_{i:i \neq j}\mathcal{A}_i $) at vertex $v$ at time $t$ by \[\mu_v(j):= \sum_{u \in V_j}P^t(u,v)X_u, \quad (\text{resp. }a_v(j):=\sum_{i:i \neq j}\mu_v(i)=\sum_{i:i \neq j}\sum_{u \in V_i}P^t(u,v)X_u).\]   As $\sum_{v \in V} \mu_v(j)=|\mathcal{A}_j|$, there can be at most one set $\mathcal{A}_j$ with $\mu_v(j)>(1-2 \delta )\bar \pi_v  $ for all $v$ (assuming $n\ge 8$; where we have used $\gd \le 1/8$ and $\sum_{v \in V} X_v \le n+n^{2/3}$). Hence we only have to treat the case that  $\mu_v(i) \le (1-2 \delta ) \bar \pi_{v}  $ for all $v$ and the case that $\mu_v(i) \le (1-2 \delta ) \bar \pi_v$ holds for some of the vertices, but not for all.   
 
Assume that  $\mu_v(i) \le (1-2 \delta ) \bar \pi_v  $ for all $v$. Then since the configuration is $\delta$-balanced
\begin{equation}
\label{re: avi}
\forall  v \in V, \quad  a_v(i)\ge (1-\delta) \bar \pi_v - \mu_v(i) \ge \delta \bar  \pi_v = \sfrac{ \delta d_{v}}{d}. 
\end{equation}  
An elementary calculation (see Lemma \ref{lem: independentBer}) now shows  that for all $v$ the probability that there is a walker not from $\mathcal{A}_i$ at $v$ at time $t$ is at least $1-e^{-a_v(i)}$. Hence trivially every walker from the set $\mathcal{A}_i$ has a chance of at least $1-e^{-a_v(i)}$ of meeting some walker not from $\mathcal{A}_i$ at time $t$. In this case, the proof is concluded using \eqref{re: avi}. 

Finally, assume that $\mu_v(i) > (1-2 \delta ) \bar  \pi_v  $ for some $v$ but not for all $v$. We argue that in this case, there must be some vertex $u$ such that for some absolute constant $c_1>0$
\begin{equation}
\label{eq:u}
 c_1  /d \le \mu_u(i) \le (1-2 \delta)\bar \pi_u.
\end{equation}
Once this is establish, the proof is concluded as follows. Let $u$ be as above. Since the configuration is $(\delta,t)$-balanced we have that 
$$a_u(i)\ge (1-\delta) \bar  \pi_u - \mu_u(i) \ge  \delta  \bar  \pi_u   . $$     
By \eqref{eq:u} together with the elementary Lemma \ref{lem: independentBer} below, the probability that $u$ is occupied at time $t$ by both a walker from $\mathcal{A}_i$ and also by a walker not from $\mathcal{A}_i$ is at least
\begin{equation}
\label{eq: cd4}
(1-e^{-\mu_u(i) })(1-e^{-a_u(i) }) \ge \sfrac{1}{4} \mu_u(i)a_u(i) \ge \sfrac{c_{1}}{4d}\delta \bar \pi_u  \ge \bar c /d^{1+\xi} \quad 
\end{equation}
as desired (where we have used $1-e^{-x} \ge x-\half x^2 \ge \sfrac{x}{2}$ for $x \in [0,1]$).
We now verify the existence of such $u$ satisfying \eqref{eq:u}. Let
\begin{equation}
\label{e:Didef}
D_i:=\{v:\mu_v(i) \ge (1-2\delta) \bar \pi_v \}.
\end{equation}
We argue that if $u \notin D_i$ and $(v,u) \in E $ for some $v \in D_i$, then $u$ satisfies \eqref{eq:u} (by our assumption on $\mathcal{A}_i$ the set $D_i$ is non-empty and is a proper subset of $V$). This follows from Lemma \ref{lem:neighbormass} below (whose proof is deferred to the appendix \S\ref{s:proofl4.X}; See Remark \ref{r:neighbormass} for the intuition behind this lemma), applied to the collection of walkers from $\mathcal{A}_i$ (i.e.\ take $Y_y(0)$ in Lemma \ref{lem:neighbormass} to be 0 if the walkers which are at time 0 at vertex $y$ do not belong to $\mathcal{A}_i$, otherwise take it to be the number of such walkers; Note that condition \eqref{eq: defbalanced2} from the definition of a $(\delta,t)$-balanced configuration implies that the term $Y  \exp (-c_{1} t )$ from Lemma \ref{lem:neighbormass} is  much smaller than the term $\mathbb{E}[Y_v(t)]$).

\medskip

We now prove part (ii). We now assume $G$ is regular, and so $\bar \pi_v=1$ for all $v$.  We fix some $i$. The case that  $\mu_v(i) \le 1-2 \delta $ for all $v$ is the same as before. We now treat the case that $\mu_v(i) > 1-2 \delta   $ for some $v$ but not for all $v$. We may assume that $|\mathcal{A}_i| \le \sfrac{1}{2}(n+n^{2/3}) $, as there can be at most one class which contains more than half of the walkers. Let $D_i$ be as in \eqref{e:Didef}. Observe that $|D_i| \le \frac{ |\mathcal{A}_i|}{(1-2 \delta )} \le  \sfrac{2}{3}(n+n^{2/3}) $ (as $\delta \le 1/8$)  and so $\max \{ | \partial_{E}^{\mathrm{out}}D_i|, | \partial_{E}^{\mathrm{in}}D_i| \} \ge \sfrac{c_0d}{2}$, where $   \partial_{E}^{\mathrm{out}}A:=\{(u,v) \in E : u \in A ,v \notin A \} $ and $   \partial_{E}^{\mathrm{in}}A:=\{(v,u) \in E : u \in A ,v \notin A \} $. We first treat the case that $ | \partial_{E}^{\mathrm{out}}D_i|\ge \sfrac{c_0d}{2}$.  

For each $e=(v,u)  $ let $Z_{e}$ be the indicator of the event that there is some walker from $\mathcal{A}_i$ which was at $v $ at time $t-1$ and then moved to $u$ at time $t$. Let $Z:=\sum_{e \in \partial_{E}^{\mathrm{out}}D_i  }Z_{e}$. As in the proof of part (i), by Lemma \ref{lem:neighbormass}, there exists some $c>0$ such that $\mathbb{E}[Z_e]>c/d $ for all $e \in  \partial_{E}^{\mathrm{out}}D_i$. Hence $\mathbb{E}[Z ] \ge c_1 c_{0} $. By Fact \ref{cor:dommultinomial} we have that $(Z_{e})_{e \in E } $ are negatively correlated, and thus $\mathrm{Var}(Z) \le \mathbb{E}[Z ]  $. Thus, using the one-sided Chebyshev's inequality \eqref{eq: onesidedChebyshev} we get that $\Pr[Z>0 ] \ge \sfrac{ \mathbb{E}[Z]}{1+ \mathbb{E}[Z]}  \ge c_2 c_0$. Conditioned on $Z_e=1$ for  $e=(v,u) \in  \partial_{E}^{\mathrm{out}}D_i$, as in the proof of part (i), the probability that there is a walker not from $\mathcal{A}_i$ at $u$ at time $t$ is bounded from below (as $u \notin D_i$ and the configuration is $(\delta,t)$-balanced). This concludes the proof in the case that  $ | \partial_{E}^{\mathrm{out}}D_i|\ge \sfrac{c_0d}{2}$.

Now consider the case that  $ | \partial_{E}^{\mathrm{in}}D_i|\ge \sfrac{c_0d}{2}$. For each $e=(v,u)  $ let $\widehat Z_{e}$ be the indicator of the event that there is some walker from $\cup_{j:j \neq i} \mathcal{A}_j$ which was at $v $ at time $t-1$ and then moved to $u$ at time $t$. Let $ \widehat Z:=\sum_{e \in \partial_{E}^{\mathrm{in}}D_i  }\widehat Z_{e}$. As in the proof of part (i), by Lemma \ref{lem:neighbormass}, there exists some $c>0$ such that $\mathbb{E}[\widehat Z_e]>c/d $ for all $e \in  \partial_{E}^{\mathrm{in}}D_i$. Hence $\mathbb{E}[\widehat Z ] \ge c_1 c_{0} $. By Fact \ref{cor:dommultinomial} we have that $(\widehat Z_{e})_{e \in E } $ are negatively correlated. Thus as before, using the one-sided Chebyshev's inequality \eqref{eq: onesidedChebyshev}, we get that $\Pr[\widehat Z>0 ]>c_2 c_{0} >0$. Conditioned on $\widehat Z_e=1$ for  $e=(v,u) \in  \partial_{E}^{\mathrm{in}}D_i$, as in the proof of part (i), the probability that there is a walker  from $\mathcal{A}_i$ at $u$ at time $t$ is bounded from below (as $u \in D_i$). 
\end{proof}

\begin{lemma}
\label{lem: independentBer}
Let $\xi_1,\ldots, \xi_m$ be independent Bernoulli random variables. Denote $p_i:=\E[\xi_i]$,  $S:=\sum_i \xi_i $ and $\mu=\mathbb{E}[S]$. Then $ \Pr[S=0 ]=\prod_{i=1}^m (1-p_{i}) \le \exp[\sum_{i=1}^m - p_i] = e^{-\mu}.$
\end{lemma}
\begin{lemma}
\label{lem:neighbormass}
Let $G=(V,E)$ be a finite Eulerian graph. Let $Y_v(t) $ be the number of walkers which occupy vertex $v$ at time $t$, where $(Y_v(0))_{v \in V}$ is some arbitrary deterministic initial configuration of walkers and different walkers perform independent LSRWs. Let $(v,u) \in E $. Let $Y:= \max_{x \in V } \sfrac{ d_{v} }{d_{x}}Y_x(0) $. Then
\begin{equation}    
\label{e:neighbormass00}
2d_v\mathbb{E}[Y_u(t)] \ge \mathbb{E}[Y_v(t-1)] \ge \sfrac{1}{3}(\mathbb{E}[Y_v(t)]-Y  \exp (-c_{1} t ) ).
\end{equation}
\end{lemma}
\begin{remark}
\label{r:neighbormass}
The proof of Lemma \ref{lem:neighbormass} relies on laziness in a crucial manner (in fact, this is the only use of laziness in the proof of \eqref{eq: upper}). The idea is that typically a lazy walk of length $t$ makes at least $t/3$ lazy steps (the error term $Y  \exp (-c_{1} t )$ corresponds to the total contribution to $\mathbb{E}[Y_v(t)]$ of the paths that make less than $t/3$ lazy steps). Such a walk could be transformed into many walks of length $t-1$ with the same end-point by deleting one of the lazy steps. The ``cost" of this operation is at most a constant factor (this uses the fact that there are at least $\Omega(t)$ lazy steps). In fact, some algebra yields that it is at most $\sup_{k \in \N :k \in [t/3,t]} \binom{t}{k}/\binom{t}{k-1} \le 3$.  
\end{remark}
\subsection{Proof of proposition \ref{prop: larged}}
\label{s: larged}
{\em Proof:} By \cite[Lemma 2.4]{PS} there exists some $C>0$ such that for  $t:=\lceil C |E||V|  \rceil $   $$P^{t}(x,y) \ge \pi_y/2, \quad \text{for all }x,y \in V. $$
Hence for every pair of walkers with some arbitrary initial positions we have that the probability that they are at the same position at time $t$ is at least $\sum_{x \in V}(\pi_x/2)^2 \ge \frac{1}{4|V|}(\sum_x \pi_x)^2 = \frac{1}{4|V|}$. We may assume that the total number of walkers $Y$ is greater than $3|V|/4$, as the probability that this fails is exponentially small in $|V|$. Under this assumption, it follows from Lemma \ref{lem: 2ndmoment} that there exists some constant $\alpha>0$ such that for every collection of $1+\frac{1}{2}\lceil 3|V|/4 \rceil$ walkers with some arbitrary initial conditions, the probability that the first walker met at time $t$ one of the other walkers in the collection is at least $\alpha$. It follows that (if $Y>\frac{3|V|}{4}$) the configuration of walkers at each fixed time $s$, deterministically, has the $(\alpha,t)$-merging property. The proof is now concluded using Proposition \ref{p:derof1.1}.     \qed

\subsection{Proof of \eqref{eq: upper} for regular $G$}
\label{s:reglineard}
In this subsection we prove \eqref{eq: upper} when $G$ is regular. This is done by combining Proposition \ref{prop: main2} below, which refines Proposition \ref{prop: main}, with Proposition \ref{p:derof1.1}. 

Let  $G=(V,E) $ be a connected $d$-regular $n$-vertex Eulerian digraph.   Let 
\begin{equation}
\label{e:t*G}
t_*(G):= \max \{\hat t_*, \lceil \sfrac{2}{c_1}\log \log n \rceil \}, \text{ where } \hat t_*:=\inf\{t: \max_{x,y}|P^t(x,y)- \sfrac{1}{n} |\le \sfrac{1}{ 3 \cdot 2^8 \log n}\} \end{equation}
where  $c_1>0$ is as in Lemma \ref{lem:neighbormass}. Note that as the walk is lazy, the probability it stays put for $\lfloor  \log_{2} (c\log n) \rfloor $ steps is at least $\sfrac{1}{c \log n}$ and thus $t_*(G) \le C\hat t_*$. Recall that by Fact \ref{fact: decay}   \[ t_*(G) \le C' ( \log n)^2.\] The choice of the constant $3 \cdot 2^8  $ was made so that by Corollary \ref{cor:LDforfuture}, provided that (*) holds for $\delta=1/8$, we have that the configuration of walkers is $(\sfrac{1}{8},t_*(G))$-balanced w.p.\ at least $1-2n^{-2} $. The choice of the constant $\sfrac{2}{c_1}  $ was made to ensure that for a  $(\sfrac{1}{8},t_*(G))$-balanced configuration we have in Lemma \ref{lem:neighbormass} that $\mathbb{E}[Y_v(t)] - Y  \exp (-c_{1} t ) \ge \frac{1}{2}\mathbb{E}[Y_v(t)]  $.  

 In this subsection we improve Proposition \ref{prop: main} as follows. 
\begin{proposition}    
\label{prop: main2}
Let $G=(V,E)$ be a connected  $d$-regular $n$-vertex Eulerian digraph satisfying (*) for $\delta=\sfrac{1}{8}$.   Assume that a configuration of walkers is fully $(\sfrac{1}{8},t_*(G))$-balanced  and that the   total number of walkers it contains is smaller than $n+n^{2/3}$. Then there exists some absolute constant $\bar c >0 $ such that the following hold.
\begin{itemize}
\item[(i)]
The configuration is globally $(\bar c/d,\bar c,t_*(G))$-merging.
\item[(ii)] If $|\{(u,v) \in E : u \in A,v \notin A  \text{ or }u \notin A,v \in A \}| \ge c_0 d $ for every set $A \subset V$ such that $|A| \le \sfrac{2}{3} (n+n^{2/3}) $ then the configuration is marginally $(\bar c c_{0},t_*(G))$-merging.
 \end{itemize}
\end{proposition}
\begin{proof}
 We first prove part (i). Consider some configuration $(X_v)_{v \in V}$ of walkers  with at most $n+n^{2/3}$ walkers and a vertex-respecting partition of the walkers into disjoint sets   $\mathcal{A}_1,\ldots,\mathcal{A}_m$. Again, we may think of the current time as being time 0. Let $a_u(i)$ (resp.\  $\mu_u(i)$) be the expected number of walkers not from (resp.\ from)   $\mathcal{A}_i$ which occupy $u$ at time $t_{*}(G)$. First consider the case that there is a set $I \subseteq [m] $ of size at least $(m-1)/2 $ such that for all $i \in I $ we have that $\mu_u(i) \le \sfrac{3}{4}$ for all $u$. Then from the proof of Proposition \ref{prop: main} we see that for each $i \in I $ the probability that some walker from $\mathcal{A}_i$ has met by time $t_*(G)$ some walker not from $\mathcal{A}_i$ is at least $c_3>0$. By Lemma \ref{lem: 2ndmoment} we have that with probability at least $c_3/4 $, there are at least $c_3 |I| $ sets $\mathcal{A}_i$ with $i\in I$ such that  some walker from $\mathcal{A}_i$ has met by time $t_*(G)$ some walker not from $\mathcal{A}_i$.   

Now consider the case that there is a set $J \subseteq [m] $ of size at least $(m-1)/2 $ such that for all $j \in J $ we have that $ \mu_u(j) \ge \sfrac{3}{4}  $ for some $u \in V $ but not for all $u \in V$. For each $j \in J $, as before, let $D_j:=\{v: \mu_u(j) > \sfrac{3}{4}\}$. Let \[F_j:=\{v \in D_j: \text{there exists some }u \notin D_j \text{ such that }(u,v) \in E \}.\] For every $j \in J$ fix some $v_j \in F_j$.
Observe that since the configuration is fully $(\sfrac{1}{8},t_*(G))$-balanced we have that the sets $(F_j)_{j \in J}$ are disjoint, and so $(v_j)_{j \in J}$ are distinct.

 Let $Z_j $ be the indicator of the event that some walker from $\mathcal{A}_j$  and also some walker not from $\mathcal{A}_j$ occupying $v_j$ at time $t_*(G)$. As in the proof of Proposition \ref{prop: main} we have that $\mathbb{E}[Z_j] \ge c_3 /d $. By Fact \ref{cor:dommultinomial} we have that $(Z_j)_{j \in J}$ are negatively correlated (indeed, for $i \neq j$ knowing that $Z_i=1$ decreases  the chances of both the requirements for $Z_j=1$). Hence by the one sided Chebyshev's inequality \eqref{eq: onesidedChebyshev}
$\Pr[\sum_{j \in J}Z_j > \sfrac{ c_3 }{2d} |J| ] \ge c_4>0 $ as desired.

The proof of part (ii) is identical to the proof of part (ii)
of Proposition \eqref{prop: main2}.
\end{proof}
 
\subsection{Deterministic initial configuration}
\label{s:deter}
Let $G=(V,E)$ be a  connected $d$-regular $n$-vertex Eulerian digraph. Consider the case that initially there is exactly one walker at each site and that the walkers perform independent LSRWs. Let $Y_v(s)$ be the number of particles occupying $v$ at time $s$. Let $t_*(G)$ be as in \eqref{e:t*G}. Assume $n$ is sufficiently large so that $\sfrac{1}{n} \le \sfrac{1}{ 3 \cdot 2^9 \log n} $ (which by the definition of  $t_*(G)$  implies that $\max_{x,y}P^{t_*(G)}(x,y) \le \sfrac{1}{ 2^9 \log n}$). We show that at each given time the configuration is fully $(\sfrac{1}{8},t_*(G) )$-balanced w.p.\ at least $2n^{-4}+n^{- \log \log n } $. From this, as in the proof of \eqref{eq: upper}, we get that $\whp$ $\SC(G) \le C d t_*(G) \log n\le C' d(\log n)^{3} $. 

The proof of the following proposition is  similar to that of Lemma \ref{lem: Poissum}. The main difference is that now the number of walkers at distinct vertices are not independent and thus one has to rely on \eqref{e:negaff} in order to bound $\mathbb{E}[\exp(\la Z_v(s))]$ from above.
\begin{proposition}
\label{prop:deterministic}
 Let $Z_v(s):=\sum_{u \in V }Y_u(s)P^{t_*(G)}(u,v) $. In the above setup we have that  
\begin{equation}
\label{e:laplaceZvs}
\forall s \ge 0, \quad  \Pr[Z_{v}(s) \notin ( \sfrac{7}{8},\sfrac{10}{8}) ] \le2n^{-4}. \end{equation}
\begin{equation}
\label{e:laplaceZvs'}
\forall s \ge  0, \quad  \Pr[Y_{v}(s) \ge \log n ] \le  n^{- \log \log n }. \end{equation}
\end{proposition}
\noindent \emph{Proof:}
 Let $v \in V$ and $s \ge0$ We first prove \eqref{e:laplaceZvs}. Let $\la \in [-2^9 \log n,2^9 \log n] $ to be determined later. By Fact \ref{cor:dommultinomial} and  \eqref{e:negaff}   we have that
\begin{equation}
\label{e:laplaceZvs2}
\mathbb{E}[\exp (  \la Z_v (s))]= \mathbb{E} \prod_{u \in V }\exp (  \la Y_u(s)P^{t_*(G)}(u,v)) \le \prod_{u \in V }\mathbb{E}[\exp (  \la Y_u(s)P^{t_*(G)}(u,v))].  \end{equation}

Note that $\exp (  \la Y_u(s)P^{t_*(G)}(u,v))= \prod_{y \in V} \exp (\la\xi_y(u,s) P^{t_*(G)}(u,v))]$, where $\xi_y(u,s)$ is the indicator of the event that the walker which initially occupied $y$ is at $u$ at time $s$. Then   \begin{equation}
\label{e:xiyus}
\mathbb{E}[\exp (  \la Y_u(s)P^{t_*(G)}(u,v))]=\prod_{y \in V} \mathbb{E}[\exp (  \la\xi_y(u,s) P^{t_*(G)}(u,v))].\end{equation}
Let $p_*:=\max \{ P^{t_*(G)}(x,y):x,y \in V \} \le \sfrac{1}{ 2^9 \log n} $. As $|\la|p_* \le 1 $ and $e^{x}-1 \le x+ \sfrac{x^2}{1+1_{x \le 0}}$ for $x \in [-1,1]$ we get that for all $y,u \in V $ we have that 

\[\mathbb{E}[\exp (  \la\xi_y(u,s) P^{t_*(G)}(u,v))]-1 \le P^s(y,u)(e^{ \la P^{t_*(G)}(u,v) }-1) \]
\[ \le P^s(y,u)[\la P^{t_*(G)}(u,v)+\sfrac{(\la P^{t_*(G)}(u,v))^2}{1+1_{\la \le 0}} ] \le P^s(y,u)\la P^{t_*(G)}(u,v)[1+\sfrac{\la p_{*}}{1+1_{\la \le 0}} ]. \]
Using $1+x \le e^x $, substituting the above in \eqref{e:xiyus} and using \eqref{e:laplaceZvs2}, together with the fact that the uniform distribution is stationary (and thus $\sum_y P^s(y,u)=1=\sum_u  P^{t_*(G)}(u,v) $) we get that  
 \[\mathbb{E}[\exp (  \la Z_v (s))] \le \exp ( \la +\sfrac{\la^{2} p_*}{1+1_{\la \le 0}}).  \]

Taking $\la=-\sfrac{1}{8p_*} $ and noting that $\Pr[Z_v(s) \le \sfrac{7}{8} ]=\Pr[\exp (  \la Z_v (s)) \ge e^{\sfrac{7 \la}{8}} ] $ we see that
\[\Pr[Z_v(s) \le \sfrac{7}{8} ]= \le \mathbb{E}[\exp (  \la Z_v (s))] e^{-\sfrac{7 \la}{8}} \le \exp (\la/8 +\sfrac{\la^{2} p_*}{2}])=\exp (-\sfrac{1}{2^{7}p_*} ) \le n^{-4}. \]
Taking $\la=\sfrac{1}{8p_*} $  and noting that $\Pr[Z_v(s) \ge \sfrac{10}{8} ]=\Pr[\exp (  \la Z_v (s)) \ge e^{\sfrac{10 \la}{8}} ] $ yields
\[\Pr[Z_v(s) \ge \sfrac{10}{8} ] \le \mathbb{E}[\exp (  \la Z_v (s))] e^{-\sfrac{10 \la}{8}} \le \exp (- \la/4 +\la^{2} p_*])=\exp (-\sfrac{1}{2^{6}p_*} ) \le n^{-8}. \]  
We now prove \eqref{e:laplaceZvs'}. It suffices to prove the following. Let $\xi_1,\xi_2,\ldots,\xi_m $ be independent Bernoulli r.v.'s. Let $p_i:=\mathbb{E}[\xi_i]$. Assume that $\sum_i p_i=1 $. Let $L>0$ and $\la =\log L $.  Then
\[\Pr[S> 1+L ] \le L^{-(L+1)}, \] 
where $S:=\sum_i \xi_i $. Indeed $\mathbb{E}[e^{\la \xi_i}]=1+p_i(e^{\la}-1) \le \exp ( p_i(e^{\la}-1))   $. As $\mathbb{E}[e^{\la S}]=\prod_i \mathbb{E}[e^{\la \xi_i}] $, \[\Pr [ S> 1+L ]\le\mathbb{E}[e^{\la S}] e^{- \la (1+L)}=\exp (e^{\la}-1) e^{- \la (1+L)} \le \exp(- (L+1) \log L ). \quad \text{\qed} \]    
\section{General lower bounds on the social connectivity time}
\label{s: lower}
Clearly we can bound $\SC$ from below by the minimal time by which every walker has met at least one other walker. This motivates the following definition.
\begin{definition}
We say that a walker remained \textbf{\emph{isolated}} for time $t$  if this walker has not met any other walkers up to
and including time $t$. We say that a walker is \textbf{\emph{lazy}} by time $t$ if this walker has not left her initial position by time $t$. We denote the event that there is an isolated lazy walker at vertex $v$ by time $t$ by $\mathrm{IL}_v(t)$. 
\end{definition}
 
Recall that by elementary properties of the Poisson distribution, conditioned on the total number of walkers being $m$, each of them independently starts at a random initial position chosen according to the stationary distribution. We denote the corresponding probability and expectation for the aforementioned initial condition by $\PP^{(m)} $ and $\mathbb{E}^{(m)}$.  In what comes, it will sometimes be convenient to condition on the total number of walkers.

\begin{proposition}
\label{p:mlower} Let $G=(V,E)$ be a finite Eulerian digraph of average degree $d$ and minimal degree $d_{\mathrm{min}}$. Let $m \in \N$. Let $Z(t):=\sum_{v \in V} \mathrm{IL}_v(t)$ be the number of isolated lazy walkers at time $t$. Let $Y(t)$ be the number of  walkers which remained isolated for time $t$. Then
\begin{equation}
\label{e:Mcdiam}
\mathbb{P}^{(m)}[Y(t)=0 ] \le \exp(-\sfrac{2(\mathbb{E}^{(m)}[Y(t)] )^2 }{m(t+2)^2} ).
\end{equation}
\begin{equation}
\label{e:ml1}
\mathbb{E}^{(m)}[Z(t)] \ge m \pi \{v: \bar \pi_v \le 2 \} 2^{-t} (1-\sfrac{2(t+1)}{|V|})^{m-1}\ge \sfrac{d_{\mathrm{min}}}{2d}m2^{-t}(1-\sfrac{2(t+1)}{|V|})^{m-1}.
\end{equation}
In additional, if $G$ is a regular graph then
\begin{equation}
\label{e:ml2}
\mathbb{E}^{(m)}[Y(t)] \ge m (1- \sfrac{(2t+1)/|V|}{ \min_{v \in V} \sum_{i=0}^t P^{2i}(v,v)} )^{m-1}. 
\end{equation}
\end{proposition}
\begin{proof}  Let $f(\mathbf{x}_1,\ldots,\mathbf{x}_m)$ be the number of walkers which remained isolated for time $t$ when for all $i \in [m]$ the $i$-th walker performs the walk $\mathbf{x}_i $. Let $\mathbf{x}:=(\mathbf{x}_1,\ldots,\mathbf{x}_m)$ and let $\mathbf{x}' $ be $\mathbf{x} $ with its $i$th co-ordinate replaced by some other walk $\mathbf{x}_i'$. It is easy to see that $|f(\mathbf{x})-f(\mathbf{x}')| \le t+2 $. In other words, given the walks performed by the rest of the walkers, the walk performed by the $i$-th walker can change the number of walkers which remained isolated for time $t$ by an additive term whose absolute value is at most  $t+2$. Thus \eqref{e:Mcdiam} follows by the McDiarmid's inequality \eqref{e:Mcstatement}.

Let $A:=\{v: \bar \pi_v \le 2 \}$. The second inequality in \eqref{e:ml1} follows by noting that  $|A| \ge |V|/2 $ and so $ \pi (A) \ge |A| \min_{v \in V}\pi_v \ge  \sfrac{d_{\mathrm{min}}}{2d}$. We now prove the first inequality in \eqref{e:ml1}. As each of the $m$ walkers w.p.\  $\pi \{v: \bar \pi_v \le 2 \}$ has its initial position in $A$,  it suffices to show that for each $v \in A$ the conditional probability that $\mathrm{IL}_v(t)=1$ given that the first walker's initial position is $v$ is at least $  2^{-t} (1-\sfrac{2(t+1)}{|V|})^{m-1}$. Indeed this holds as the probability of the first walker staying put up to time $t$ is $2^{-t}$, while the probability that no other walker visits $v$ by time $t$ is
\[(1- \mathbb{P}_{\pi}[T_v \le t ] )^{m-1} \ge (1- \mathbb{E}_{\pi}[|\{i \le t :X_i =v \} |  ] )^{m-1}=(1-\sfrac{2(t+1)}{|V|})^{m-1},  \]   
where $(X_i)_{i \ge 0}$ is a LSRW with  $X_0 \sim \pi$ and $\PP_{\pi}$ and $\mathbb{E}_{\pi}$ are the corresponding probability and expectation (where we have used $\pi_v \le 2/|V|$ since $v \in A$).

  We now prove \eqref{e:ml2}. Let  $\mathbf{X}:=(X_i)_{i \ge 0}$ and   $(\widehat X_i)_{i \ge 0}$ be independent LSRWs with   $X_0 \sim \pi$  and   $ \widehat X_0 \sim \pi$. Let $N(s):=|\{i \le s:X_i=\widehat X_i \}|$ be the number of times the two walks collide by time $s$. Since conditioned on the walk that a certain walker performs, the events that each of the other $m-1$ walkers does not collide with it by time $t$ are independent with the same probability (used in the first equality in \eqref{e:jensencond}), by Jensen's inequality we get that   
\begin{equation}
\label{e:jensencond}
\mathbb{E}^{(m)}[Y(t)]=m\mathbb{E}[ \mathbb{E}[1_{N(t)=0} \mid \mathbf{X} ]^{m-1}  ]\ge m(\mathbb{E}[ \mathbb{E}[1_{N(t)=0} \mid \mathbf{X} ]  ])^{m-1} = m \mathbb{P}[N(t)=0]^{m-1} .
\end{equation}
Now $ \mathbb{P}[N(t) \ge 1] \le \sfrac{\mathbb{E}[N(2t)]}{\mathbb{E}[N(2t) \mid N(t) \ge 1  ]} \le \sfrac{(2t+1)/|V|}{\min_{v \in V,i \le t} \mathbb{E}[N(2t)-N(i-1) \mid X_i=\widehat X_i=v   ]} \le \sfrac{(2t+1)/|V|}{ \min_{v \in V} \sum_{j=0}^t P^{2j}(v,v)}  $, where we have used regularity  to argue that $\pi $ is the uniform distribution and hence $\mathbb{E}[N(2t)]=\sum_{i =0 }^{2t}\mathbb{P}[X_i=\widehat X_i] = (2t+1)\sum_{v \in V} \pi_v^2=(2t+1)/|V| $, and reversibility to argue that $\mathbb{E}[N(2t)-N(i-1) \mid X_i=\widehat X_i=v   ]=\sum_{j=0}^{2t-i} \sum_u P^{j}(v,u)P^{j}(v,u) = \sum_{j=0}^{2t-i} P^{2j}(v,v) $, as $\sum_u P^{j}(v,u)P^{j}(v,u)=\sum_u P^{j}(v,u)P^{j}(u,v)=P^{2j}(v,v)$.    
\end{proof}
Equations \eqref{eq: lower} and \eqref{e:reglower2} follow from the following corollary, in conjunction with the estimate $\mathbb{P}[\Pois(n) \notin [\sfrac{n}{2},2n] ] \le 2 e^{-n/16}$, which follows from \eqref{eq: poisconcentration}.
\begin{corollary}
\label{cor:lowerm}
Let $G=(V,E)$ be a connected $n$-vertex  Eulerian digraph of average degree $d$ and minimal degree $d_{\mathrm{min}}$. Let $m \in [\sfrac{n}{2},2n] $. Let \[\kappa_t:=\min_{v \in V}\sum_{i=0}^tP^{2i}(v,v)  \quad \text{and} \quad   s_{\alpha}:=\min \{t: \sfrac{ t+1}{\kappa_{t+1}} > \sfrac{\alpha}{16}  \log n \}     .\] Let $Y(t)$ be the number of  walkers which remained isolated for time $t$. If $t \le  \sfrac{1}{36}\log n $  then
\begin{equation}
\label{e:Mcdiam2}
\mathbb{P}^{(m)}[Y(t)=0 ]\le \exp[-cn \left( \sfrac{d_{\mathrm{min}}}{d} \right)^2 e^{-20t} ] .
\end{equation}
In additional, if $G$ is a regular graph then for all $\alpha \in (0,1)$ we have that
\begin{equation}
\label{e:mc3}
\mathbb{P}^{(m)}[Y(s_{\alpha})=0 ] \ge \exp(- cn^{1- \alpha } (\log n)^{-4} ).    
\end{equation}
Finally, for the usual Poisson setup we have that
\begin{equation}
\label{e:mc4}
\mathrm{P}[Y(t)=0 ] \le \max_{m \in [\sfrac{n}{2},2n]} \mathbb{P}^{(m)}[Y(t)=0 ]+\mathbb{P}[\Pois(n) \notin [\sfrac{n}{2},2n] ].
\end{equation}
\end{corollary}
\begin{proof}
Equation \eqref{e:Mcdiam2} follows from \eqref{e:Mcdiam} and \eqref{e:ml1} together with a little algebra. We leave the details as an exercise. Similarly,
equation \eqref{e:mc3} follows from \eqref{e:Mcdiam} and \eqref{e:ml2} together with a little algebra, after noting that by Fact \ref{fact: decay} $s_{\alpha} \le s_{1} \le C (\log n)^2 $, for all $\alpha \in (0,1]$. Finally \eqref{e:mc4} follows by conditioning on the total number of walkers in the l.h.s..
\end{proof}

Recall the notation from \S\ref{sec: preliminaries} (in particular that $\bar \pi : = |V| \pi $ that  $p((\gamma_0,\cdots,\gamma_t)):=\prod_{i=1}^tP(\gamma_{i-1},\gamma_i) $ and that $q(\gamma):=\bar \pi_{\gamma_0}p(\gamma) $ and that $\Gamma_t$ is the collection of all walks of length $t$). The following proposition follows straightforwardly from Fact \ref{fact: thinning} (about Poisson thinning).
\begin{proposition}
\label{prop: mainlower}
Let $G=(V,E)$ be a finite connected Eulerian  digraph. Let  $\gamma=(\gamma_{0},\ldots,\gamma_{t}) \in \Gamma_t$. Then, given that $X_{\gamma}=1$, the conditional distribution of the number of walkers that the walker who performed the path $\gamma$ met by time $t$ is $\mathrm{Poisson}(a_{\gamma})$, where
\[ a_{\gamma}:=\sum_{\gamma' \in \Gamma_t :\gamma' \neq \gamma , \exists i, \gamma_i=\gamma'_i}q(\gamma') \le-q(\gamma)+ \sum_{i=0}^{t} \bar \pi_{\gamma_{i}}.\]
 In particular,
\begin{equation}
\label{eq: IL}
\begin{split}
\forall v, & \quad \Pr[\mathrm{IL}_v(t)] \ge 2^{-t} \bar \pi_v e^{-(t+1) \bar \pi_v }, \\
\forall v \neq u, & \quad \Pr[\mathrm{IL}_v(t) \cap \mathrm{IL}_u(t)]=\Pr[\mathrm{IL}_v(t)]\Pr[\mathrm{IL}_u(t)]e^{m_{u,v}(t)},
\end{split}
\end{equation}
where \[m_{u,v}(t+1):= \sum_{\gamma  \in \Gamma_{t+1} : u,v \in \gamma, \gamma_0 \notin \{v,u  \} }q(\gamma) \le t \bar \pi_v \mathbb{P}_v[T_u \le t]+t \bar \pi_u \mathbb{P}_u[T_v \le t].\].
\end{proposition}
\noindent
\emph{Proof:}
For the inequality, note that the probability that there is one walker at $v$ at time $0$ and that this walker is lazy by time $t$ is $ \bar \pi_v 2^{-t}e^{- \bar \pi_v}$. The conditional probability that no other walker visited $v$ by time $t$ is  $\mathbb{P}[\mathrm{Pois}(\sum_{\gamma \in \Gamma_t:v \in \gamma, \gamma_0 \neq v}q(\gamma))=0] \ge \mathbb{P}[\mathrm{Pois}(t \bar \pi_v)=0] \ge e^{-t \bar \pi_v}$.

We now prove the  second line of \eqref{eq: IL}. Let \[\Gamma_1:=\{\gamma  \in \Gamma_t : v,u \in \gamma,\, \gamma_0 \notin \{v,u  \}\},\]  \[\Gamma_2:=\{\gamma  \in \Gamma_t : v \in \gamma ,\, u \notin \gamma,\, \gamma_0 \neq v \}\quad \text{and} \] \[\Gamma_3:=\{\gamma  \in \Gamma_t : u \in \gamma ,\, v \notin \gamma,\, \gamma_0 \neq u \}.\] Denote $b_i:= \sum_{\gamma \in \Gamma_i}q(\gamma)$ (where $i=1,2,3$).
Then $b_1=m_{u,v}(t)$ and by the above reasoning $$\Pr[\mathrm{IL}_v(t)]=2^{-t}\bar \pi_v e^{- \bar \pi_v}e^{-(b_1+b_2)}, \quad \Pr[\mathrm{IL}_u(t)]=2^{-t}\bar \pi_u  e^{- \bar \pi_u}e^{-(b_1+b_3)}. $$
$$\Pr[\mathrm{IL}_v(t) \cap \mathrm{IL}_u(t)]=(2^{-t}\bar \pi_v e^{- \bar \pi_v})(2^{-t} \bar \pi_u e^{- \bar \pi_u})e^{-(b_1+b_2+b_{3})}=\Pr[\mathrm{IL}_v(t)]\Pr[\mathrm{IL}_u(t)]e^{b_{1}}. \quad \text{\qed} $$

\begin{theorem}
\label{thm: lower2}
Let $G$ be an $n$-vertex connected graph.
Let $\mu:=\sum_{v: \bar \pi_v \le 2 } \bar \pi_v$ (recall that $\bar \pi=n \pi $) and $a_t:= \mu 2^{-t}e^{-2(t+1)} $. Then for all $t \ge 0$ we have that
\begin{equation}
\label{eq: mudeltaaz}
\Pr[\cup_v \mathrm{IL}_v(t) ] \ge 1- 7t/ \sqrt{a_{t}}.
\end{equation}
Hence, if there exists some $\delta>0$ such that $\mu \ge c_1 n^{\beta} $ for some $\beta>0$, then for $s_{\beta}:=\max \{ \lfloor \tilde c_{\beta} \log_{} n-2\log \log n  \rfloor -1,0\} $, where $\tilde c_{\beta}:=\frac{\beta}{2 \log (2e^{2}) }$, we have that \[\Pr[\cup_v \mathrm{IL}_v(s_{\beta}) ] \ge 1-12n^{-\beta/4} .\]\end{theorem}
\begin{proof}
We may assume that $7t \le \sqrt{a_t}  $. Let $A:=\{v: \bar \pi_v \le 2 \}$. Let $\nu_{i}(x,y):=\sum_{j=1}^i P^j(x,y) $.  Consider $$B_a:=\{b \in A: \bar \pi_a \nu_{t}(a,b)\ge \bar \pi_b /(2 \sqrt{a_t} ) \}.$$ For every $a \in A$ we have that  $\sum_{b \in A } \nu_{t}(a,b)\le t $  and so \[\sum_{b \in B_a}\bar \pi_b \le 2 \bar \pi_a \sqrt{a_t} \sum_{b \in B_a}\nu_{t}(a,b)\le 2 \bar \pi_{a} t \sqrt{a_t}.\] This (together with  $7t \le \sqrt{a_t}  $, which implies that $t \sqrt{a_t}>1 $) implies that there exists some $D \subseteq A $ such that
\begin{equation}
\label{eq: DsubA1}
\sum_{a \in D}\bar \pi_a \ge \mu/(1+2t \sqrt{a_t} )\ge \mu/(3t \sqrt{a_t} ) \quad \text{and} 
\end{equation}
\begin{equation}
\label{eq: DsubA2}
 \bar \pi_a\nu_{t}(a,b)< \max \{ \bar \pi_a, \bar \pi_b\} /(2 \sqrt{a_t} ), \text{ for all distinct }a , b \in D.
\end{equation}
By \eqref{eq: IL} and \eqref{eq: DsubA1}, $\mathbb{E}[Z] \ge (a_t / \mu) \sum_{a \in D}\bar \pi_a \ge \sqrt{a_t}/(3t)$, where $Z:=\sum_{a \in D}Z_a $ and $Z_a:=1_{\mathrm{IL}_a(t)}$. For all distinct $a,b \in D$, we have that $2\bar \pi_a t\nu_{t}(a,b) \le 4t/\sqrt{a_t} \le 4/7 $. Thus by \eqref{eq: DsubA2}  $$\exp(2\bar \pi_a t\mathbb{P}_a[T_b \le t])-1  \le 4\bar \pi_a t\mathbb{P}_a[T_b \le t] \le 2\max \{ \bar \pi_a, \bar \pi_b\} t/ \sqrt{a_t} \le 4t/ \sqrt{a_t},     $$  and so by \eqref{eq: IL} $$\mathrm{Cov}(Z_a,Z_b)/( \mathbb{E}[Z_a]\mathbb{E}[Z_b]) \le  \exp(2\bar \pi_a t\mathbb{P}_a[T_b \le t])-1 \le 4t/ \sqrt{a_t}.$$
Hence $L:=\sum_{a \neq b, a,b \in D}\mathrm{Cov}(Z_a,Z_b) \le (\mathbb{E}[Z])^2(4t/ \sqrt{a_t})$. Finally, since $\mathrm{Var} Z=\mathbb{E}[Z]+L$ by Chebyshev's inequality we have that $\Pr[Z=0] \le \frac{\mathrm{Var} Z}{(\mathbb{E}[Z])^2}  \le \frac{1}{\mathbb{E}[Z]}+4t/ \sqrt{a_t} \le 7t/ \sqrt{a_t}.  $  
\end{proof}
The next example demonstrates that the assertion of Theorem \ref{thm: lower2} is quite sharp.
\begin{example}
Denote $L_n:=\lceil \log^{10}n \rceil$. Consider a $\lceil n/L_n \rceil $-clique, such that each vertex in the clique is the center of a star of size $L_n$ (where all stars are disjoint). Then one can show that $\SC \le C \log \log n $ $\mathrm{w.h.p.}$ (and if the walkers perform non-lazy SRWs we would get that $\SC \le C$ $\mathrm{w.h.p.}$). The reason for this is that the expected number of walkers whose initial position is not in a center of a star is at most $L_n^2$, while at each vertex of the clique there are $\mathrm{Pois}(\mu_n)$ walkers for some $\mu_n = L_n (1 \pm o(1) )$. An easy calculation shows that $\mathrm{w.h.p.}$ after one step all of the walkers whose initial positions are in the clique are in the same class. 
\end{example}
{\em Proof of Theorem \ref{thm: reg}:} Equation \eqref{e:reglower2} follows from Corollary \ref{cor:lowerm}. We now prove \eqref{e:reglower1}.
Let $G=(V,E)$ be a connected $n$-vertex $d$-regular graph. Consider the SN model on $G$ in which the walks have some fixed holding probability $p \in [0,1)$. Denote the transition matrix of the corresponding walks by $P$.  Recall that $\Gamma_t$ is the collection of all walks on $G$ of length $t$ and that for $\gamma=(\gamma_0,\ldots,\gamma_t) \in \Gamma_t$, $p(\gamma):=\prod_{i=0}^{t-1}P^{i}(\gamma_i,\gamma_{i+1}) $. Fix $t=t_n= \lfloor \log n - 6\log \log n \rfloor-1$. Assume $t \ge 0$ as otherwise there is nothing to prove. For every $\Gamma' \subseteq \Gamma_t$ let $m(\Gamma'):=\sum_{\gamma \in \Gamma'}p(\gamma)$. For every $\gamma \in \Gamma_t$ and $0 \le i \le t$ let 
 \[B_{\gamma,i}:=\{v \in V: \sum_{j =0 }^tP^{|i-j|}(\gamma_j,v) \ge (t+1)^{-3} \}, \]
\[\mathrm{Bad}_{\gamma}:=\{\tilde \gamma \in \Gamma_t : \exists i\in \{0,1,\ldots,t\} \text{ such that } \tilde \gamma_i \in B_{\gamma,i}  \}. \]
Clearly $|B_{\gamma,i}| \le (t+1)^{4}$ for all $\gamma \in \Gamma_t$ and $i \in \{0,1,\ldots,t\}$, and thus \begin{equation}
\label{eq: reg1}
 m(\mathrm{Bad}_{\gamma}) \le \sum_{i \le t}\sum_{v \in B_{\gamma,i}}m(\{\tilde \gamma \in \Gamma_t :\tilde \gamma_i=v \}) \le \sum_{i \le t}|B_{\gamma,i}|  \le (t+1)^5 .
\end{equation}
Fix some order $\prec$ on $\Gamma_t$. Recall that $X_{\gamma}$ denotes the number of walkers who performed the walk $\gamma$.  We sequentially expose $X_{\gamma}$ for some paths $\gamma \in \Gamma_t$ according to the following procedure. Assuming that we have  already exposed $A \subset \Gamma_t $ so that the following holds: for every $\gamma \in A$ with $X_{\gamma}>0$ we have that $\mathrm{Bad}_{\gamma} \cap \{\gamma' \in A: \gamma \prec \gamma' \} =\{\gamma \}$ (where $\gamma \prec \gamma'$ indicates that $\gamma'$ is larger in the ordering than $\gamma$). In the next stage we expose $X_{\gamma'}$ for the minimal $\gamma' \in \Gamma_t \setminus  B(A) $, where  $B(A):=\cup_{\gamma \in A:X_{\gamma}>0}\mathrm{Bad}_{\gamma} $. In the following stage we apply the same rule, with the set $A$ replaced by the set $A \cup \{\gamma'\}$. At the end of this procedure we obtain a collection $\mathcal{W}=\{\gamma^1,\ldots,\gamma^{|\mathcal{W}|}\}$ (resp.~$\mathcal{N}$) of all $\gamma$'s for which we exposed that $X_{\gamma}>0$ (resp.~$=0$), where the indices are taken so that $i<j$ iff $\gamma^i$ is before $\gamma^j$ in the ordering. The collection of all $\gamma \in \Gamma_t$ for which $X_{\gamma}$ was not exposed is precisely $\cup_{\gamma \in \mathcal{W}:X_{\gamma}>0}\mathrm{Bad}_{\gamma}\setminus \{\gamma\}=\Gamma_t \setminus (\mathcal{W} \cup \mathcal{N})$. 

\medskip
Let $\mathcal{N}_i$ be the collection of all $\gamma \in \mathcal{N}$ which were exposed in between $\gamma^{i-1}$ and $\gamma^{i}$ (and $\mathcal{N}_1$ is the collection of all $\gamma$'s exposed prior to $\gamma^1$). 
For every $1 \le i \le |\mathcal{W}|$, the probability that  $m(\mathcal{N}_i)>(t+1)^5$ is at most $\Pr[\mathrm{Pois}((t+1)^5)=0]= e^{-(t+1)^5}$. Hence by \eqref{eq: reg1}
\begin{equation}
\label{eq: reg2}
\Pr[|\mathcal{W}| < n/(2(t+1)^5)] \le ne^{-(t+1)^5}=o(1). 
\end{equation}
We now condition on $\mathcal{W}=W$ and $\mathcal{N}=N$ for some arbitrary $W,N \subset \Gamma_t$ so that (i) $ |\mathcal{W}| \ge n/[2(t+1)^5] $ and (ii) $\mathrm{Bad}_{\gamma}\cap \{\gamma' \in W: \gamma \prec \gamma' \}=\{\gamma\}$, for all $\gamma \in W$.
Observe that, given $(\mathcal{W},\mathcal{N})=(W,N)$, the joint distribution of $(X_{\gamma})_{\gamma \in \Gamma_t \setminus (W\cup N)}$ is the same as their unconditional joint distribution. For $\gamma=(\gamma_0,\ldots,\gamma_t),\gamma'=(\gamma_0',\ldots,\gamma_t')  \in W$ let \[\Gamma_{\gamma}:=\{ \tilde \gamma \in \Gamma_t \setminus (N \cup \{\gamma\}):\exists i \in \{0,1,\ldots,t\} \text{ such that }\tilde \gamma_i=\gamma_i \} \subseteq \mathrm{Bad}_{\gamma} ,\]
\[\Gamma_{\gamma,\gamma'}:=\{\tilde \gamma \in \Gamma_t : \exists i,j\in \{0,1,\ldots,t\} \text{ such that }\tilde \gamma_i=\gamma_i \text{ and }\tilde \gamma_j=\gamma'_j \} \subseteq \mathrm{Bad}_{\gamma}\cap \mathrm{Bad}_{\gamma'}. \]
 Note that $m_{\gamma}:=m(\Gamma_{\gamma}) \le \sum_{i=0}^tm(\{\tilde \gamma \in \Gamma_t :\tilde \gamma_i=\gamma_i\}) = t+1$. Denote \[D_{i,j}(\gamma,\gamma'):=\{\tilde \gamma \in \Gamma_t:\tilde \gamma_i=\gamma_i,\, \tilde \gamma_j=\gamma'_j,\, \forall k <i\, \, \tilde \gamma_k \neq \gamma_k,\, \forall r<j\, \, \tilde \gamma_r \neq \gamma_r' \} \}\] By construction   $\gamma' \notin \mathrm{Bad}_{\gamma} $,  for all $\gamma,\gamma' \in W$ such that $\gamma \prec \gamma' $ and so $\sum_{j= 0}^tP^{|i-j|}(\gamma_j,\gamma'_i) < (t+1)^{-3}$ for all $0 \le i \le t$. Hence by reversibility 
\[m_{\gamma,\gamma'}:=m(\Gamma_{\gamma,\gamma'})=\sum_{0 \le i<j \le t}m(D_{i,j}(\gamma,\gamma'))+m(D_{i,j}(\gamma',\gamma)) \le \sum_{0 \le i,j \le t}P^{|i-j|}(\gamma_i,\gamma'_j)  \le (t+1)^{-2}. \]Finally, for each $\gamma \in W $ set $Z_{\gamma} $ to be the indicator of the event that $\sum_{\gamma \in \Gamma_{\gamma}}X_{\gamma}=0$. Set $Z=\sum_{\gamma \in W} Z_{\gamma}$. Observe that (using $m_{\gamma} \le t+1$ for all $\gamma \in W$) \[\mathbb{E}[Z]  \ge |W|e^{-(t+1)}> \sfrac{n}{2(t+1)^5}(n^{-1}\log^6n)> \half \log n \]  while similarly to \eqref{eq: IL} for all $\gamma,\gamma' \in W$ we have that \[\mathrm{Cov}(Z_{\gamma},Z_{\gamma'})/(\mathbb{E}[Z_{\gamma}] \mathbb{E}[Z_{\gamma'}])=e^{m_{\gamma,\gamma'}}-1 \le e^{1/(t+1)^{2}}-1=:\delta_n. \] Hence 
\[\mathrm{Var}(Z) \le \mathbb{E}[Z]+\delta_n(\mathbb{E}[Z])^2,  \]
which by Chebyshev implies that $\Pr[Z=0]< \frac{1}{\mathbb{E}[Z]}+\delta_n $. \qed

\section{The cycle}
\label{s: cycle}
In this section we consider the case that $G$ is the $n$-cycle $\mathrm{C}_n$. In this section we assume that at time 0 there is precisely one walker at each site and that the walkers perform independent continuous-time random walks with jump rate 1.

 We first prove \eqref{eq: cycleupper}. We fix $s=s_n:=C_1 \log^2 n $. By a union bound over the $n$ edges of the cycle, in order to deduce that $\Pr[\mathrm{SC}(\mathrm{C}_n) \ge s] \le C/n $ it suffices to show that for every pair of neighboring vertices $u,v$, the probability that the walkers who started at those vertices do not have a path of acquaintances by time $s$ is at most $Cn^{-2}$.   

 Consider some fixed edge $e=\{u,v\}$. It partitions the vertices into 2 sides (segments)  $A_u,A_v$  of size roughly $n/2$, according to the identity of the end-point to which they are closer to (break a tie arbitrarily). Let $e'$ be the other edge, apart from $e$, which connects $A_u$ and $A_v$. A walker starting from either $u$ or $v$ is extremely unlikely to cross $e'$ in $s$ time units.

 Let $\mathrm{CROSS}_{s}(e,w) $ be the event that the walker who started at vertex $w$ did not cross $e'$ by time $s$ and that her position at time $s$ is at the opposite side of $e$ compared to $w$. 

We now argue that for every $(w,w') \in A_u \times A_v $ the event  $\mathrm{CROSS}_s(e,w) \cap \mathrm{CROSS}_s(e,w') $ ``forces" the walkers who started at $u$ and $v$ to have a path of acquaintances by time $s$ (which uses only walkers whose starting positions are in $\{u,v,w,w' \}$, possibly a subset of this set). Indeed,
the continuous-time setup eliminates the possibility of two walkers swapping positions without meeting. This implies that on the event  $\mathrm{CROSS}_s(e,w) \cap \mathrm{CROSS}_s(e,w') $ it must be the case that by time $s$ the walkers whose starting positions are $w$ and $w'$, resp., have met, and that each of the walkers whose starting positions are $u$ and $v$, resp., must have met at least one of the previous pair of walkers (in fact, this is the case for any pair of vertices $u',v'$ lying in the segment between $w$ and $w'$ which contains $e$). 

This reduces the proof of \eqref{eq: cycleupper} into the following simple calculation: If $C_1$ is taken to be sufficiently large then  $\Pr[\mathrm{CROSS}_s(e,w)] \ge 0.4 $  for every edge $e=\{u,v\}$ and  vertex $w$ of distance at most $10 \log n $ from $e$. Indeed, starting from such a vertex $w$, if $C_1$ is sufficiently large then the probability of crossing $e$ by time $s$ is arbitrarily close to $1/2$ (in fact it is arbitrarily close to 1), and    for all sufficiently large $n$  the probability that $e'$ is crossed by time $s$ is at most $1/20$ (this probability tends to 0 as $n$ to infinity). Hence (using $(\frac{12}{20})^{10 \log n} \le n^{-2} $) $$\Pr[\mathrm{CROSS}_s(e)] \ge 1-Cn^{-2}, \text{ where } \mathrm{CROSS}_s(e):= \bigcup_{w \in A_u, w' \in A_v}\mathrm{CROSS}_s(e,w) \cap \mathrm{CROSS}_s(e,w') .$$ 

This concludes the proof of \eqref{eq: cycleupper}. We now prove \eqref{eq: cyclelower}.   We fix $t=t_{n}=c_1 \log^2 n $ for some constant $c_1$ to be determined later. Observe that if there are at least 2 edges which no walker crossed by time $t$, then deterministically $\SC(\mathrm{C}_n)>t$. Let $\mathrm{NC}_t(e)$ be the event that the edge $e \in E(\mathrm{C}_n) $ was not crossed by any walker by time $t$. We argue that if $c_1$ is taken to be sufficiently small, then for all sufficiently large $n$ and every edge $e$ it holds that  
\begin{equation}
\label{eq: NCe}
\Pr[\mathrm{NC}_t(e) \mid \mathrm{F} ] \ge c_3 n^{-1/4},
\end{equation}
where  $\mathrm{F}$ is the event that no walker performed at least $(\sqrt{ n}/2)-2$ steps by time $t$. We now explain how \eqref{eq: cyclelower} can be derived from \eqref{eq: NCe}. Denote $m:=\lfloor \sqrt{n} \rfloor $. Let $(e_i)_{i=1}^{m}$ be a collection of edges which are all at distance at least $\sqrt{n}$ from one another.  Then (by a union bound over the $n$ walkers) $1-\Pr[\mathrm{F}] \le C_2ne^{-cn^{1/4}} \le C_{2}'e^{-c'n^{1/5}}$. Finally, note that conditioned on $\mathrm{F}$, the events $(\mathrm{NC}_t(e_i))_{i=1}^m$ become independent, and so by  \eqref{eq: NCe} the conditional probability that none of them occur is at most $(1-c_{3}n^{-1/4})^m \le e^{-c'_{3}n^{1/4}} $.

We now prove \eqref{eq: NCe}. It is not hard to see that in order to prove prove \eqref{eq: NCe} it suffices to show that for every $k$ if we put one particle at each site of $\N:=\{1,2\ldots\}$ and the particles perform independent continuous-time SRW on $\Z$ with jump rate 1, then the probability that no walker reached the origin by time $k$ is at least $e^{-M \sqrt{k} } $, for some absolute constant $M$. Denote the probability of the previous event by $p_k$. Let $(S_{r}^{i})_{r \ge 0}$ be a SRW on $\Z$  starting at $i\in \N$. Then by the reflection principle (see e.g.~\cite[\S2.7]{levin2009markov}; the proof in continuous-time is analogous)
\begin{equation}
\label{e:aikaik}
a_i(k):=\mathbb{P}\left[S_{r}^{i}>0, \, \forall r \in [0, k] \right]=\mathbb{P}\left[S_{k}^{0}\in
\{-i+1,\ldots,i \}\right], \quad \text{for all }i \in \N.
\end{equation}
 
Write $b:=\lfloor 4 \sqrt{k} \rfloor $. We argue that there exist constants $c_5,c_6,C_3>0$ such that
\begin{equation}
\label{eq: LCLT}
a_i(k) \ge c_5i/ b, \quad \text{for all } 1 \le i \le b.
\end{equation}
\begin{equation}
\label{eq: LDregime}
\begin{split}
a_i(k) & \ge 1-2 \exp [- \half i^2/k+i^4/4k^3] \ge 1-2e^{-i^{2}/(8k)} \ge e^{-C_3e^{-i^{2}/(8k)}},  \text{ for }b<i \le 1.2k. 
\\ a_i(k) & \ge 1-e^{-c_6 i} \ge e^{-C_3e^{-c_6 i}}, \quad \text{ for all }i > 1.2k.
\end{split}
\end{equation}
The first line follows from the local CLT. The first inequality in the third line follows from the fact that $1-a_i(k) \le \mathbb{P}[\mathrm{Pois}(k) \ge i]$. The last inequalities in the second and third lines follows from the fact that for every $0<c<1$ there exists $C>0$ such that  $1-x \ge e^{-Cx}$,  for all $x \le c$ (applied to $x=2e^{-i^{2}/(8k)}$ for $i>b$). The first inequality in the second line is obtained by noting that for discrete-time SRW  $(\tilde S_{r})_{r \in \Z_+}$  starting at the origin we have that $\mathbb{E}[e^{\lambda \tilde S_{r}}]=\left(\frac{1}{2}e^{\lambda}+\frac{1}{2}e^{-\lambda}\right)^{r} \le e^{\lambda^2 r/2} $ (by comparing Taylor expansion coefficients). Hence,   $$\mathbb{E}[e^{\lambda  S_{k}}] \le \sum_r \Pr[\mathrm{Pois}(k)=r]e^{\lambda^2 r/2}=\exp [k(e^{\lambda^2 /2}-1)] \le \exp [k(\lambda^2/2+(\lambda^2/2)^2)] , $$
as long as $\lambda^2/2 \le 1$ (using $e^b-1 \le b+b^2$ for all $b \in [-1,1 ]$). We set $\lambda=i/k$, so that indeed $\lambda^2/2 \le 1 $ for $i \le 1.2k$. Finally, by \eqref{e:aikaik}, Markov's inequality and our choice of $\lambda$ $$1-a_i(k) \le 2 \mathbb{P}[S_k \ge i] \le 2 \mathbb{E}[e^{\lambda  S_{k}}]e^{-\lambda i}\le 2 \exp [-(i^2/2k)+i^4/(4k^3)].  $$ 

We are now in a position to conclude the proof. By \eqref{eq: LCLT} and Stirling's approximation  $$ \prod_{i \le b }a_i(k) \ge c_5^{b  }b!/b^{b} \ge \sqrt{b}(c_5/e)^b \ge e^{-C_4 \sqrt{k}}. $$
Denote $b':= \lfloor \sqrt{k} \rfloor$. Finally, using \eqref{eq: LDregime} it is not hard to show that 
$$\prod_{i > b }a_i(k) \ge \prod_{\ell: 4 \le \ell \le 1.2b' }a_{\ell b'}^{b'}(k)\prod_{\ell> 1.2b' }a_{\ell}(k) \ge e^{-C_{5} \sqrt{k}}$$
(we leave the details to the reader; Alternatively, since $a_i(k)$ are uniformly bounded away from 0 for $i>b$, we get that $\prod_{i > b }a_i(k) \ge e^{-C \sum_{i>b}(1-a_i(k)) } $ and the reasoning in the following remark yields that $ \sum_{i>0 }(1-a_i(k))= \sum_{i>0}\Pr_i[T_0 \le k] \le C' \sqrt{k} $). We are done as $p_k:= \prod_{i >0 }a_i(k)$. \qed
\begin{remark}
\label{rem: Poistrick}
Consider the usual Poisson setup. The proof of the upper bound on $\SC(\mathrm{C}_n)$ is similar to the case of one walker per site at time 0. However, as we now explain, the proof of the lower bound on $\SC(\mathrm{C}_n)$  becomes much simpler. While this was already noted in the introduction after Theorem \ref{thm: reg}, we present below an alternative argument, as it will be used in the analysis of Examples \ref{rem: extermal} and \ref{ex:nonregular}.

 For simplicity, we present the argument in the discrete-time setup. Let $v$ be a vertex in $\mathrm{C}_n$.  Denote $\nu_t:=\sum_{i=0}^{t}P^i(v,v)$ (this sum is independent of $v$). For $t \le n^2$ we have that $\nu_t=\Theta (\sqrt{t+1}) $. By Fact \ref{fact: thinning}, the total number of walkers to reach vertex $v$ by time $t$ has a Poisson distribution with mean (by reversibility)
\begin{equation*}
\begin{split}
\mu_t & :=\sum_{u}\mathbb{P}_{u}[T_v \le t]= \sum_{u}\sum_{j=0}^t \mathbb{P}_{u}[T_v = j] \le\frac{1}{\nu_t} \sum_{u}\sum_{j=0}^t \mathbb{P}_{u}[T_v = j] \sum_{i=0}^{2t-j}P^i(v,v) \\ &  \le \frac{C}{\sqrt{t+1}}\sum_{u}\sum_{i=0}^{2t}P^i(u,v) = \frac{C(2t+1)}{\sqrt{t+1}}.
\end{split}
\end{equation*}
Thus if $t+1< (\frac{\log n}{8C})^2$ we get that the probability that no walker reached $v$ by time $t$ is at least $e^{-\mu_t} \ge n^{-1/4}$. This implies the lower bound on $\SC(\mathrm{C}_n)$ as in the proof above.
\end{remark}
\section{Expanders}
\label{s: expanders}
In this section we study the case that $G$ is a $d$-regular $\la$-expander. It is not difficult to extend the results to the case $G$ is an expander of maximal degree $d$.
LSRW on a regular $\lambda$-expander $G=(V,E)$ mixes rapidly in the following sense 
 \begin{equation}
 \label{eq: mix}
 \max_{x,y \in V}|P^t(x,y)-\pi_y|\le (1-\lambda)^{t}, \quad \text{ for all }t.
\end{equation}
Indeed $\max_{x \in V}|P^t(x,x)-\pi_x|\le (1-\lambda)^{t}$
follows from the spectral decomposition of $P^t(x,x)$ along with the non-negativity of the eigenvalues of $P$.  However, for LSRW on a regular graph $\max_{x,y \in V}|P^t(x,y)-\pi_y|=\max_{x\in V}|P^t(x,x)-\pi_x|$, since in general (cf.~\cite[(2.2)]{spectral}
for even $t$ and \cite[p.~135]{levin2009markov} for moving from even $t$
to odd $t$ using laziness)
\begin{equation*}
\label{eq: p(x,y)}
\max_{x,y} |P^{t}(x,y)-\pi_y| \le \sqrt{ \max_{u,v}(\pi_u/\pi_v)} \max_x |P^t(x,x)-\pi_x|  .
\end{equation*} 
\subsection{Proof of Theorem \ref{thm: et1}}
We  argue that Theorem \ref{thm: et1} follows from Theorem \ref{thm: et2}. This follows at once from Corollary \ref{cor: prey} below. Indeed, after a (``giant") class of walkers $\mathcal{A} $ of size at least $n/6$ emerges, by Corollary \ref{cor: prey} (using a union bound over the walkers not in $\mathcal{A}$, together with the concentration of the total number of walkers around $n$), w.p.~at least $1-C_{2}n^{-1}$, the additional amount of time until every walker not from $\mathcal{A} $ will meet some walker from $\mathcal{A} $ is at most  $\lceil
C \lambda^{-1}\log 
n \rceil $. \qed
\begin{lemma}
\label{lem: hitinexpander}
Let $G$ be a connected $d$-regular $n$-vertex $\lambda$-expander. Then there exists $C>0 $ such that for $t:= \lceil C \lambda^{-1}\log
n \rceil  $ 
and every $v, u_{1},u_{2},\ldots,u_{t}$ sequence of vertices,
the probability that a LSRW started at $v$  visits some $u_i$ at time $i$ for some $1 \le i \le t$ is at least $ \frac{1}{8} \min ( \frac{ \lambda t}{n},1)$.
\end{lemma}
\begin{proof}
Let $(X_s)_{s \ge 0}$ be a LSRW on $G$ started at $v$. Let $Y_{i}=1_{X_i=u_i} $ and $Y:=\sum_{i=1}^t Y_i$, where $t:= \lceil C \lambda^{-1}\log n \rceil $, for some constant $C$, to be determined shortly. By \eqref{eq: mix}, if $C$ is sufficiently large, $\mathbb{E}[Y] \ge t/(2n) $, whereas $\mathbb{E}[Y_iY_{i+j}] \le\mathbb{E}[Y_i]P^{j}(u_i,u_{i+j}) \le \mathbb{E}[Y_i](n^{-1}+(1-\lambda)^{j})$ and so $\mathbb{E}[Y^2] \le 2 \mathbb{E}[Y](\lambda^{-1}+t/n )   $. Finally, $\Pr[Y>0] \ge \frac{ \mathbb{(E}[Y])^{2}}{\mathbb{E}[Y^2]} \ge  \frac{\mathbb{E}[Y]}{2(\gamma^{-1}+\frac{t}{n} )} \ge \frac{1}{8} \min \{ \frac{ \lambda t}{n},1 \} $.
\end{proof}
\begin{corollary}
\label{cor: prey}
Let $G=(V,E)$ be a connected $d$-regular $n$-vertex $\lambda$-expander. Place $\ell:= \lceil n/6 \rceil $ walkers arbitrarily on $G$ along with one extra walker. Assume all walkers perform independent LSRWs on $G$ simultaneously. Then there exist some  constants $C,n_{0}$, independent of $G,d,\lambda$ and the initial locations of the walks, such that if $n \ge n_{0}$
the extra walker meets at least one of the other $\ell$ walkers by time $t:= \lceil C \lambda^{-1}\log n \rceil$ with probability at least $1-1/n^2 $ (as usual, meeting means visiting the same vertex at the same time).   
\end{corollary}
\begin{proof}
First,  condition on the walk that the \textcolor{red}{extra walker} performed by time $t$, denoted by $\gamma$. Then use Lemma \ref{lem: hitinexpander} and independence to obtain the desired estimate for the conditional probability (uniformly, for all possible values of $\gamma$). Finally, average over $\gamma $.
\end{proof}

\subsection{Proof of Theorem \ref{thm: et2}}
\noindent
\emph{Proof:} We adapt a technique from \cite[Proposition 3.1]{alon} to our setup. Let $t \le n$ to be determined later. Set $s=s(t,\la):=\lceil8(t+2)/\lambda \rceil $. Call a class of walkers at time $s$ \textbf{good} if the collection of walkers from this class occupy at time $s$ at least $t$ vertices.
By Corollary \ref{cor: Stage1} below (if $t$ is sufficiently large), w.p.\ at least $1-\exp(-c_{1}n/d^{4 (s+1 )})$ at time $s$  there exists a collection containing at most $n/(2t)$ good classes such that the walkers belonging to the union of these classes occupy at least $n/2$ vertices at time $s$. We conditioned on this event and on the identity of the good classes from this collection and on the positions at time $s$ of the corresponding walkers. We claim, that if $t$ is taken to be sufficiently large, then w.p.~at least $1-\exp(-c_2 n)$ there is no way of splitting this collection into 2 (disjoint) collections, $\mathcal{A},\mathcal{B}$, such that
\begin{itemize}
\item[(i)]
The walkers belonging to the union of the classes in  $\mathcal{A}$ (resp.~$\mathcal{B}$), occupy at time $s$ at least $n/6$ vertices.
\item[(ii)]
No walker from (the union of the classes in)   $\mathcal{A}$ met a walker from $\mathcal{B}$ by time $s+r$, where $r:=\lceil C \lambda^{-1}
\rceil $ for some sufficiently large constant $C$. 
\end{itemize}
Indeed, this follows from Lemma \ref{lem: stage2} below which asserts that for any choice of $\mathcal{A},\mathcal{B}$ satisfying (i), the probability that (ii) holds is at most $e^{-c_2n}$, by a union bound over all  $ \le 2^{n/(2t)}$ such partitions (and picking $t$ such that  $t \ge \lceil 1/c_2 \rceil$; Recall that in the considered collection of good classes  there are at most $n/(2t)$ classes). \qed
\begin{lemma}
\label{lem: step1}
Let $G=(V,E)$ be a connected $d$-regular $n$-vertex $\lambda$-expander. Consider the SN model on $G$. Let $\mathcal{W}_u(s)$ denote the walkers who occupy vertex $u$ at time $s$. Write $u \Leftrightarrow_{s} v $ if there exists a pair of walkers $(w,w') \in \mathcal{W}_u(s) \times \mathcal{W}_v(s) $ who met each other by time $s$.  Let $A_{u,r}(s)$ be the event that $|\{v:u \Leftrightarrow_{s} v\}| \ge r$. Denote $s:=\lceil8(t+2)/\lambda \rceil $. Then there exists an absolute constant $c>0 $ such that if $ s \le n $ then $$\Pr[ A_{u,t }(s)] \ge (1-e^{-1})-e^{-ct}. $$
\end{lemma}
\begin{proof}
 We have that $\Pr[ |\mathcal{W}_u(s)|>0]=1-e^{-1}$. Throughout the proof we condition on   $|\mathcal{W}_u(s)|>0$. Moreover, we fix some walker $w \in \mathcal{W}_u(s)$ and condition on the walk $\mathbf{w}=(\mathbf{w}(0),\ldots,\mathbf{w}(s)=u)$ that she performed by time $s$. By averaging, it suffices to show that the conditional probability that $A_{u,t }(s)$ fails is at most $e^{-ct}$ for some constant $c>0$ independent of $\mathbf{w}$. Below, all probabilities and expectations are the conditional ones, given $\mathbf{w}$.

Let $(X_k)_{k \ge 0}$ be a LSRW on $G$. \textcolor{red}{We denote the law of  $(X_k)_{k \ge 0}$ given $X_0=v$ by $\mathbb{P}_v$.} Let $\tau_{\mathbf{w}}:=\inf \{k \in [0,s]  :X_k=\mathbf{w}(s-k)\}$.
By Fact \ref{fact: thinning} and reversibility, the (conditional) distribution (given $\mathbf{w}$) of the number of walkers not from  $\mathcal{W}_u(s)$  that $w$ met by time $s$ has a Poisson distribution with mean  $\mu_{\mathbf{w}}:= \sum_{v:v \neq u}\mathbb{P}_{v}[\tau_{\mathbf{w}} \le s]$. We argue that  $\mu_{\mathbf{w}} \ge 4t $.
Indeed, let $b_{v,\mathbf{w}}:=\sum_{i
=0}^{s}P^i(v,\mathbf{w}(s-i))$. By reversibility and \eqref{eq: mix} $ \mu'_{\mathbf{w}}:= \sum_{v:v \neq u }b_{v,\mathbf{w}}=s- b_{u,\mathbf{w}}\ge s-\lambda^{-1}-1$. Again by \eqref{eq: mix} and reversibility, for all $v$ $$b_{v,\mathbf{w}}=\sum_{i \le s}\mathbb{P}_{v}[\tau_{\mathbf{w}} =i] \sum_{j
=0}^{s-i}P^j(\mathbf{w}(s-i),\mathbf{w}(s-i-j)) \le \sum_{j=0}^{s-1} \max_{x,y}P^j(x,y)\mathbb{P}_{v}[\tau_{\mathbf{w}}
\le s ]$$ $$ \le\mathbb{P}_{v}[\tau_{\mathbf{w}} \le s ](\lambda^{-1}+1) \le 2 \lambda^{-1} \mathbb{P}_{v}[\tau_{\mathbf{w}}
\le s ].$$ Summing over all $v  \in V \setminus \{ u\}$ we get that  $\mu_{\mathbf{w}} \ge (\lambda/2) \mu'_{\mathbf{w}} \ge (\lambda/2)(s-\lambda^{-1}-1) \ge 4t $, as desired.  

Let $Z_v$ be the number of walkers in $\mathcal{W}_v(s)$ who met $w$ by time $s$. Then (conditioned on $\mathbf{w}$) $(Z_v)_{v:v \neq u}$ are independent Poisson random variables such that $\mathbb{E}[ Z_{v}] \le 1 $ for all $v$ and $\sum_{v:v
\neq u} \mathbb{E}[ Z_{v}] \ge 4t$. Finally, by (the elementary) Lemma \ref{lem: 1-e-1} below we have that $\mathbb{E}[1_{Z_v \ge 1}] \ge (1-e^{-1})\mathbb{E}[ Z_{v}]$ for all $v$, and so $\sum_{v:v
\neq u }\mathbb{E}[1_{Z_v \neq 0 }]>(1-e^{-1})4t>2t $.   By Bernstein's inequality \eqref{e:berns} $\Pr[\sum_{v:v \neq u }1_{Z_v \neq 0 }\le t] \le e^{-ct}$, for some constant $c>0$, as desired.        
\end{proof}
\begin{corollary}
\label{cor: Stage1}
In the setup and notation of Lemma \ref{lem: step1}, there exist  $t_0,c_{1}>0$ such that
$$\Pr[\sum_{u \in V} 1_{A_{u,t }(s)}<n/2] \le \exp(-c_{1}n/d^{4(s+1)}), \text{ for all } t \ge t_0 \text{ such that }s=\lceil8(t+2)/\lambda \rceil \le n. $$  
\end{corollary}
\begin{proof}
Use Lemma \ref{lem: step1} together with Azuma inequality on an appropriate Doob's sequence, by exposing the value of indicators on the l.h.s.~sequentially, one at a time. The absolute value of the increments are bounded by the maximal size of a ball of radius $2s$ ($\le d^{2(s+1)} $).
\end{proof}
\begin{lemma}
\label{lem: stage2}
Let $G=(V,E)$ be a connected $d$-regular $n$-vertex $\lambda$-expander. Let $0<\delta \le 1/2$. Let $A,B \subset V $ be two disjoint sets of size at least $\delta n$. Put one walker at each vertex of $A \cup B$ and let them perform independent LSRWs. There exists an absolute constant $c>0$ such that the probability that no walker from $A$ met a walker from $B$ by time
$r:=\lceil \lambda^{-1}\log (4/\delta^{2}) \rceil $ is at most $\exp[-c \delta^2 n]$. 
\end{lemma}
\noindent
\emph{Proof:}
Let $\mathcal{W}_A,\mathcal{W}_B$ be the collection of walkers occupying $A$ and $B$, respectively, at time 0 (as indicated in the statement of the lemma, at each vertex there is exactly one walker at time 0). Let $D_{t}$ be the set of vertices which are occupied at time $t$ by at least one walker in $\mathcal{W}_B$. Using Lemma \ref{lem: 1-e-1} below and the fact that $\Pr[u \in D_t ]=P^t(u,B) \le 1 $ for all $u$, it is easy to verify that  $\mathbb{E}[|D_t|] \ge |B|(1-e^{-1})$ for every $t$. 
Thus by applying Azuma inequality to an appropriate Doob's sequence (by exposing the positions at time $t$ of the walkers in $\mathcal{W}_B$ sequentially, one at a time), we get that 
\begin{equation}
\label{eq: Dtlarge}
\Pr[|D_t|  < \delta n/4] \le e^{-c_1 \delta n }
\end{equation}
(the absolute values of the increments of this martingale are bounded by 1). Denote the uniform distribution on $A$ by $\pi_A$ and by $\Pr_{\pi_{A}}^r$ the distribution at time $r$ of LSRW with $X_0 \sim \pi_A$. If some walker from $\W_A$ is at $D_r$ at time $r$ then this walker meets a walker from $\mathcal{W}_B$ at time $r$. Given that $D_r=D$, the conditional probability that no walker from $\mathcal{W}_A$ is at $D_r$ at time $r$  is at most
\[\prod_{a \in A}(1-P^{r}(a,D)) \le \exp [-\sum_{a \in A}P^{r}(a,D) ]=\exp [-|A|\Pr_{\pi_{A}}^r (D)
].\] To conclude the proof, we show that $\Pr_{\pi_{A}}^r (D) \ge \delta/8 $ for all $D \subseteq V$ such that $|D| \ge \delta n /4$, which implies the assertion of the Lemma using \eqref{eq: Dtlarge}. 

Let  $\pi$ be the uniform distribution on $V$. Recall that for every $p \ge 1$ the $\ell_p$ distance between a pair $\mu,\nu$ of distributions on $V$ is defined as $\|\mu-\nu \|_{p,\pi}:=[\sum_{v \in V}\pi(v)(\frac{\mu(v)-\nu(v)}{\pi(v)})^p]^{1/p} $ and $\|\mu-\nu \|_{\infty,\pi}:=\max_{v \in V}|\mu(v)-\nu(v)| $.
Then  \[\|\pi_{A}-\pi \|_{2,\pi} \le \|\pi_{A}-\pi \|_{\infty,\pi}\le n/|A|. \] By Jensen's inequality (first inequality below) and  the  Poincar\'e inequality (second inequality below) satisfied by the spectral gap (cf.~\cite[Lemma 3.26]{aldous}) \[\|\Pr_{\pi_{A}}^r -
\pi \|_{1,\pi}  \le \|\Pr_{\pi_{A}}^r -
\pi \|_{2,\pi} \le (1-\lambda)^r \|\pi_{A}-\pi \|_{2,\pi} \le \sfrac{n}{|A|}e^{-\lambda
r} \le \delta/4 ,\] where the last inequality follows from the choice of $r$. As $\|\Pr_{\pi_{A}}^r -
\pi \|_{1,\pi}= \sum_v |\Pr_{\pi_{A}}^r
(v)-\pi(v)|=2 \max_{D \subset V} \pi(D)- \Pr_{\pi_A}^r(D)  $),  it follows that if $|D| \ge \delta n/4$ then \[\Pr_{\pi_A}^r(D) \ge \pi(D)-\sfrac{1}{2}  \|\Pr_{\pi_{A}}^r -
\pi \|_{1,\pi}= (|D|/n) -\delta/8\ge \delta /8 . \quad \text{\qed} \]
\begin{lemma}
\label{lem: 1-e-1}
For all $x \in [0,1]$, $(1-e^{-x}) \ge (1-e^{-1})x$. In particular, for all $x \in [0,1]$,  if $Y \sim \mathrm{Pois}(x)$ or $Y=\sum_{i=1}^{m}\xi_i$, where $\xi_1,\ldots,\xi_m$ are independent Bernoulli r.v.'s with $\sum_{i=1}^{m}\mathbb{E}[\xi_i]=x$, then (by Lemma \ref{lem: independentBer}) $\Pr[Y \ge 1] \ge 1-e^{-\mathbb{E}[Y]} \ge (1-e^{-1})\mathbb{E}[Y]$. 
\end{lemma}

\section{Concluding remarks, conjectures and open problems}
\label{s: open}

\subsection{Allowing different communities of walkers}

It is not hard to see that the proof of \eqref{eq: upper} can also be extended to the following setup in which there are $m$ ``communities" of walkers. Let $m \le n $. Let  $k_1,\ldots,k_m > 0 $ be such that $\sum_{i=1}^mk_i=n$. Let $G_1,\ldots,G_m $ be connected graphs on the same vertex set $V$ of size $n$. Denote the stationary distribution of LSRW on $G_i$ by $\pi^i$. At time 0 there are $\mathrm{Pois}(k_i \pi_v^i )$ walkers at vertex $v$ performing independent LSRWs on $G_i$ (for all $1 \le i \le m$ and $v \in V $, independently). As usual, when two walkers reach the same vertex at the same time they become acquainted, even if they walk on different graphs. The bound on $\SC$  depends in this case on $\max_{ i \le m}d^{(i)}$, where $d^{(i)}$ is the average degree of $G_i$.  
\subsection{Dependence on the maximal degree and the vertex-transitive setup}
\begin{definition}
\label{def: VT}
We say that a bijection $\phi:V \to V$ is an automorphism
of a graph $G=(V,E)$ if $\{u , v \} \in E $ iff $\{ \phi(u) , \phi(v)\} \in E$. A graph $G$ is said to be \emph{vertex-transitive} if the action of its automorphisms group  $\mathrm{Aut}(G)$   on its vertices is transitive (i.e.~$\{\phi(v):\phi \in \mathrm{Aut}(G) \}=V $ for all $v$).
\end{definition}
\begin{conjecture}
\label{con: cycle}
There exists some $C>0$ such that for every connected $n$-vertex graph $G$ of average degree $d$, we have   $\sfrac{ \SC(G)}{( \log n)^2} \le C(d^21_{G \text{ is non-regular}}+d  1_{G \text{ is regular but not vertex-transitive}}+1)$ w.p.\ at least $1-C/n$.
\end{conjecture}
\begin{conjecture}[Monotonicity in the number of walkers]
\label{con: monotone}
Denote by $\SC(G,k)$ the social connectivity time when we have $k$ walkers, each starting at a vertex chosen from the stationary distribution $\pi$ independently of the rest of the walkers. Then for every vertex-transitive graph $G$ we have that $\mathbb{E}[\SC(G,k)]$ is monotonically
non-increasing in $k$ for $k \ge 2$.
\end{conjecture}

\begin{open}
\label{open: vt}
Let $G=(V,E)$ be a finite connected vertex-transitive graph. Let $o \in V$. Consider the SN model on $G$. Consider
\begin{itemize}
\item
$t(G)$ -  the minimal integer satisfying $t(G) \ge  \log |V|  \sum_{i=0}^{t(G)}P^i(o,o)$, \item
$\tau_1$ -  the minimal time in which all classes are of size at least 2, 
\item
and $\tau_2$ -  the minimal time by which every vertex has been visited by at least one walker.
\end{itemize}
Is there an absolute constant $C$ such that $$\max \{\mathbb{E}[\SC(G)],\mathbb{E}[\tau_1],\mathbb{E}[\tau_2],t(G) \} \le C \min \{\mathbb{E}[\SC(G)],\mathbb{E}[\tau_1],\mathbb{E}[\tau_2] ,t(G)\} \, \text{?}$$ 
\end{open}
\begin{conjecture}
\label{con: concentration}
Consider a sequence $G_n$ of vertex-transitive graphs of increasing sizes. Then in the notation of Open problem \ref{open: vt},
\[\SC(G_n)/\mathbb{E}[\SC(G_n)] \to 1 \quad \text{and} \quad  \tau_i(G_n)/\mathbb{E}[\tau_i(G_n)] \to 1\,\ (i=1,2) \quad \text{in distribution, as }n \to \infty. \]
\end{conjecture}
We believe that in the setup of the previous conjecture it is often the case that $[\SC(G_n)-\tau_1(G_n)]/\mathbb{E}[\SC(G_n)] \to 0 $ in distribution (e.g.~we believe this is the case for transitive expanders). However, this might not be the case for the $n$-cycle. It is plausible that Conjecture \ref{con: concentration} holds even without the transitivity assumption as long as the graphs $G_n$ are of bounded degree. Conversely, if $G_n$ is two disjoint $n$-cliques connected by a single edge, then $\frac{\SC(G_n)}{\mathbb{E}[\SC(G_n)]}$ converges in distribution to the Exponential distribution with parameter 1. 

In the last example we have that $\mathbb{E}[\SC(G_n)]=\Theta (n) $. This shows $\mathbb{E}[\SC(G_n)]$ may increase with the maximal degree. Below are two families of graphs, generalizing the previous example, which we believe to be extremal.  
\begin{example}
\label{rem: extermal}
In general, $\mathbb{E}[\SC(G)]$ may  grow linearly in $d$, even when $G$ is regular, as the following example demonstrates. Consider two disjoint $n$-cliques. Delete an edge from each clique and  connected the two cliques by two edges so that the obtained graph is  $(n-1)$-regular. By considering the time required until some walker visits both  cliques   it is not hard to show that $\mathbb{E}[\SC] \ge c n $. 
We now describe a similar construction of $d$-regular graphs which satisfy $\mathbb{E}[\SC] \ge c d \log^{2} (n/d)$.

 Fix some $2 \le d ,n$ such that $2d \le n$. Let $J_k$ be a graph obtained from the complete graph on $k$ vertices by deleting a single edge. Consider $\lceil n/d \rceil $ disjoint copies of $J_{d}$: $I_0,\ldots,I_{\lceil n/d\rceil-1 }$, where  for all $0 \le j < \lceil n/d \rceil$, $I_j$ is connected to $I_{j+1}$ ($j+1$ is defined $\mathrm{mod}\lceil n/d \rceil $) by a single edge that connects two degree $d-1$ vertices. This can be done so that the obtained graph is $d$-regular. It is not hard to verify that the obtained graph satisfies $\mathbb{E}[\SC] \ge cd \log ^{2}(n/d)=:s_{d,n}=s$ (in fact, if $d \ll n$ we have that $\whp$  $\SC \ge s $). To see this, observe that similarly to the analysis of the $n$-cycle it suffices to show that provided that $c$ from the definition of $s$ is taken to be sufficiently small, with probability bounded from below (in fact, $\whp$ when $d \ll n$) there will be multiples $j$'s such that no walker crossed the edge connecting $I_j$ with $I_{j+1}$ by time $s$. To prove this, one can imitate the analysis from Remark \ref{rem: Poistrick}. The claimed bound becomes intuitive when one considers the fact that the projection of the walk to the index of the copy of $J_d$ to which it belongs, behaves similarly to a random walk on the $\lceil n/d\rceil$ cycle with holding probability $1-\sfrac{1}{d^2}$. Thus the expected number of walkers to cross an edge connecting two copies  of $J_d$ by time $s$ is comparable to $d$ times the expected number of walkers which cross a certain edge by time $s/d$ for the $n$-cycle (one has to multiply by $d$ as $J_d$ has $d$ vertices).   
\end{example}
 We now consider a non-regular family of examples. 
\begin{example}
\label{ex:nonregular}
Consider two disjoint $n$-cliques connected by a path of length $n$. Using the fact that the number of walkers which initially occupy the $n$-path has a Poisson distribution with a constant mean, it is not hard to show that $\mathbb{E}[\SC] \ge c n^2 $. 

We now describe a similar construction whose maximal degree is $d \le 8|V| $ which satisfy $\mathbb{E}[\SC] \ge c [d \log (|V|/d)]^{2} $. 
  Take $\lceil n/d \rceil $ copies of $J_{d}$: $I_0,\ldots,I_{\lceil n/\ell\rceil -1} $ and for all $0 \le k < \lceil n/d\rceil -1 $  connect $I_k$ to $I_{k+1}$ via a path of length $d$.   

To see that  $\mathbb{E}[\SC] \ge c [d \log (n/d)]^2=t_{d,n}$, observe that with positive probability (in fact, $\mathrm{w.h.p.}$ if $d \ll n$) there will be many vertices lying in the middle of a length $d$ path (connecting  some $I_k$ to $I_{k+1}$) which will not be visited by a single walker  up to time $t_{d,n}$. We leave the details to the reader (hint: use the idea from Remark \ref{rem: Poistrick}).  
\end{example}
\subsection{Additional Open Problems}

The following conjecture concerns the existence of infinite classes of walkers after a constant number of steps in the setup of infinite graphs.
\begin{conjecture}
Let $G$ be a connected infinite bounded degree graph such that the critical density for independent
percolation on $G$ is strictly less than 1. If  the number of walkers at time 0 at each vertex $v$ has a   $ \mathrm{Pois}(\lambda d_v)$ distribution,
independently for different $v$'s (let $\mathbb{P}_{\lambda}$ be the corresponding probability measure), then for every $\lambda>0$ there exists  $ t_c(\lambda,G)>0$ such that for all $t>t_c(\lambda,G)$, $$\mathbb{P_{\lambda}}[\text{there exists an
infinite class of walkers at time }t]=1.$$
\end{conjecture}
Let $K_n$ be the complete graph on $n$ vertices. When the holding probabilities of the walks are taken to be either $q_n=0$ or $q_n=1/n$, one can show that after a single step the SN model (with $\mathrm{Pois}(1)$ walkers per site) behaves in some sense like a critical Erd\H os R\'enyi random graph, $G(n,1/n)$, while after two steps it behaves like the super-critical random graph  $G(n,2/n)$ (hence $\mathrm{w.h.p.}$ a giant class of walkers emerges after precisely 2 steps). We believe that also in the continuous-time setup there is a phase transition.
\begin{open}
Consider the SN model on $K_n$ in continuous-time (each walker has jump rate 1). What is the critical time $t_c$ for the emergence of a giant class of walkers?
\end{open}
\section*{Acknowledgements}
We thank Gady Kozma for his great contribution for some of the results presented in this paper.
We  like to thank also Riddhipratim Basu, Hilary
Finucane, Eviatar Procaccia, Allan Sly and especially Ofer Zeitouni for many useful discussions and
for reading some version of this work
carefully and providing helpful comments.

\bibliographystyle{plain}
\bibliography{SN}

\appendix
\section{Deferred proofs}

\subsection{Proof of Lemma \ref{lem: Poissum}}
\label{s:proofl4.8}
\textbf{Lemma 4.10}
Let $\xi_1,\ldots, \xi_m$ be independent Poisson random variables. Let $p_1,\ldots,p_m \in (0,1)$. Denote    $p_*:=\max p_i$,  $S:=\sum_{i=1}^m p_{i} \xi_i $ and $\mu:=\mathbb{E}[S]$. Then for all $\eps \in (0,1]$
\begin{equation*}
\begin{split}
& \Pr[S \le (1-\eps)\mu ] \le e^{-\half \mu \eps^2/p_* }.
\\ & \Pr[S \ge (1+\eps)\mu ] \le e^{-\sfrac{1}{4} \mu \eps^2/p_* }.
\end{split}
\end{equation*} 
\noindent
\emph{Proof:}
Denote $\mu_i:=\E[\xi_i]$ and $\lambda:=\eps/p_*$. As $\sum_k e^{-\lambda p_i k} \mu_i^k / k! = e^{\mu_i e^{-\lambda p_i}}$ and $\lambda p_i \le 1 $ (which implies that $e^{-\lambda p_i}-1 \le -\lambda p_i +\half (\lambda p_i)^2 $) we have that $$\mathbb{E}[e^{-\lambda p_i \xi_i}]=\exp[\mu_i(e^{-\lambda p_i}-1)]\le \exp [\mu_i(-\lambda p_i +\half (\lambda p_i)^2)]\le \exp [\mu_i(-\lambda p_i +\half \lambda^{2} p_* p_i)], $$
$$ \mathbb{E}[e^{-\lambda S}]= \prod_{i=1}^m \mathbb{E}[e^{-\lambda p_i \xi_i}] \le \exp [\sum_{i=1}^m \mu_i(-\lambda p_i + \half \lambda^{2} p_* p_i)]  =  \exp[\mu( - \lambda +\half \lambda^{2} p_*  )]. $$
Finally, by Markov's inequality and the choice of $\lambda$
$$\Pr[S \le (1-\eps)\mu ]=\Pr[e^{-\lambda S} \ge e^{-\lambda(1-\eps)\mu} ] \le \mathbb{E}[e^{-\lambda S}]e^{\lambda(1-\eps)\mu} \le e^{\mu( - \lambda \eps+ \half \lambda^{2} p_*  )}=e^{- \half \mu \eps^2/p_* }. $$
Similarly, let $a:=\eps /(2 p_*) $. Then  $\sum_k e^{a p_i k} \mu_i^k / k! = e^{\mu_i e^{a p_i}}$. Since  $a p_i \le 1 $ we have that $e^{a p_i}-1 \le a p_i +(a p_i)^2  $ and so $$\mathbb{E}[e^{a p_i \xi_i}]=\exp[\mu_i(e^{a p_i}-1)]\le \exp [\mu_i[(a p_i +(a p_i)^2)]\le \exp [\mu_i(a p_i +a^{2} p_* p_i)], $$
$$ \mathbb{E}[e^{a S}]= \prod_{i=1}^m \mathbb{E}[e^{a p_i \xi_i}] \le \exp [\sum_{i=1}^m \mu_i(a p_i +a^{2} p_* p_i)]  =  \exp[\mu(a+\lambda^{2} p_*  )]. $$
Finally, by Markov's inequality and the choice of $a$ we get that
$$\Pr[S \ge (1+\eps)\mu ]=\Pr[e^{a S} \ge e^{a(1+\eps)\mu} ] \le \mathbb{E}[e^{a S}]e^{-a(1+\eps)\mu} \le e^{\mu( - a \eps+  a^{2} p_*  )}=e^{-  \mu \eps^2/(4p_*) }. \quad \qed $$ 

\subsection{Proof of Lemma \ref{lem: mainlem}}
\label{s:proofl4.2}
\textbf{Lemma 4.2}
Let $J_k$ be a non-increasing $\Z_+$-valued random process, measurable w.r.t.~filtration $(\F_k)_{k \ge 0}$, with $J_0 \le m$, satisfying that for some $0<\alpha,p,\eps_m<1$, for each $k$ there exists some events $A_k \in \F_k$ with $\Pr(A_k) \ge 1-\eps_m$ satisfying that
\begin{equation*}
 \mathbb{E}[1_{\{J_{k+1} \le (1-p/2) J_k \}}1_{A_k}  \mid \F_k ] \ge \alpha1_{A_k}.
\end{equation*}
Let $T_0:=\inf \{k:J_k=0 \}$. Then there exists some constant $C_{\alpha,p} \le K/(\alpha p)$ such that 
$$\Pr[T_0 > \lceil C_{\alpha,p}\log m \rceil] \le m^{-2}+\eps_m (\lceil C_{\alpha,p}\log m \rceil +1). $$ 
\noindent
\emph{Proof:}
Denote the complement of $A_k$ by $B_k$. Let $\tau:=\inf \{k: 1_{B_k}=1 \} $. Define $I_k=J_k$ for $k<\tau$ and $I_k=0$ for $k \ge \tau$. Denote $L:=\lceil C_{\alpha,p}\log m \rceil$, where $C_{\alpha,p}$ shall be determined soon. Let $T_0':=\inf \{k:I_k=0 \}$. By a union bound over $\bigcup_{i=0}^{L}B_k$ it suffices to show that \[\Pr[T_0' \ge L] \le m^{-2}\] (since $\Pr[T_0 > L] -\Pr[T_0' > L] \le \Pr(\bigcup_{i=0}^{L}B_k)  $). Let  $a:=1/(1-\frac{1}{2} \alpha p )$. By \eqref{eq: conditionalexp} $M_k:=a^{k}I_k $ is a super-martingale (w.r.t.~the filtration $(\F_k)_{k \ge 0}$). In particular, for $C_{\alpha,p}:=3/\log a $ we have \[\mathbb{E}[I_{L} ] \le a^{-L}\mathbb{E}[M_0] \le m^{-3}\mathbb{E}[ J_0] \le m^{-2}.\] By Markov's inequality $$\Pr[T_0' \ge L]=\Pr[I_L \ge 1] \le \mathbb{E}[I_L] \le m^{-2}. \qed $$
  
\subsection{Proof of Lemma \ref{lem:neighbormass}}
\label{s:proofl4.X}
\textbf{Lemma 4.14} Let $G=(V,E)$ be a finite Eulerian graph. Let $Y_v(t) $ be the number of walkers which occupy vertex $v$ at time $t$, where $(Y_v(0))_{v \in V}$ is some arbitrary deterministic initial configuration of walkers and different walkers perform independent LSRWs. Let $(v,u) \in E $. Let $Y:= \max_{x \in V } \sfrac{ d_{v} }{d_{x}}Y_x(0) $. Then
\begin{equation}    
\label{e:neighbormass2}
2d_v\mathbb{E}[Y_u(t)]  \ge \mathbb{E}[Y_v(t-1)] \ge \sfrac{1}{3}(\mathbb{E}[Y_v(t)]-Y  \exp (-c_{1} t ) ).
\end{equation}   
\emph{Proof:}
The first inequality in \eqref{e:neighbormass2} is trivial. We now prove the second inequality. Let $v \in V $ and $t \in \N$. We say that a path $\gamma'$ is non-lazy if $\gamma'(i) \neq \gamma'(i+1)$ for all $i$.
Let $\hat \Gamma_{\ell}$ be the collection of all non-lazy paths of length at most  $\ell$ which terminate at $v$.  
 For a walk $\gamma$ let $\gamma'$ be its projection to its non-lazy version (obtained by omitting from $\gamma$ coordinates with the same value as that of the previous time unit).

For every  $\gamma' \in \hat \Gamma_{t} $ (resp.\   $\gamma' \in \hat \Gamma_{t-1} $) let $W_{\gamma'}(t)$ (resp.\  $W_{\gamma'}(t-1) $) be the collection of walks of length $t$ (resp. $t-1$) whose non-lazy version is $\gamma'$. Let $\hat P$ be the transition matrix of the non-lazy walk (i.e.\ $P=\half(I+\hat P)$, where $I$ is the identity matrix). Let $b_{k}:=\PP[\mathrm{Bin}(t,\half)=k]$ and $a_k:=\PP[\mathrm{Bin}(t,\half)>k]$. We first note  that the contribution to $\mathbb{E}[Y_v(t)]$ coming from walks in $\cup_{\gamma'\in \hat \Gamma_{t} \setminus \hat \Gamma_{\frac{2t}{3}}  }W_{\gamma'}(t)$ (i.e.~from walks $\gamma$'s with less than $t/3$ lazy steps) is at most
\begin{equation}
\label{e:toolittlelazy}
\sum_{k> 2t/3 }b_k \sum_{x}Y_x(0) \hat P^k(x,v) \le \sum_{k> 2t/3 }b_k \sfrac{ Y}{d_{v}}   \sum_{x} d_{x} \hat P^k(x,v)=a_{2t/3}Y \le Ye^{-c_1 t},   \end{equation}
where we have used  $\sum_{k> 2t/3 }b_k \le e^{-c_1 t} $, and $\pi \hat P^k=\pi $ to argue that $\sum_{x} d_{x} \hat P^k(x,v)=d_{v}$.

Recall (\S\ref{sec: preliminaries}) that    $p(\gamma)$ is the probability that a walker starting from $\gamma(0)$  performs the walk $\gamma$.  Let  $\gamma' \in \hat \Gamma_{\frac{2t}{3}} $.   We argue that  the total contribution to $\mathbb{E}[Y_v(t)] $ coming from walks in  $W_{\gamma'}(t)$ (which equals $\sum_{\gamma \in W_{\gamma'}(t)}p(\gamma)Y_{\gamma(0)}(0) )$ is larger than the contribution to $\mathbb{E}[Y_v(t-1)]$ coming from walks in  $ W_{ \gamma'}(t-1)$ (which equals $\sum_{\gamma \in  W_{ \gamma'}(t-1)}p(\gamma)Y_{\gamma(0)}(0) $) by a factor of at most $3$.   Since the total number of lazy steps performed  is at least $t/3$,  the ``cost" $\sfrac{\sum_{\gamma \in  W_{ \gamma'}(t)}p(\gamma)}{\sum_{\gamma \in  W_{\gamma'}(t-1)}p(\gamma)} $ of making one less lazy step is bounded from above by \[\sup_{k \in \N :k \in [t/3,t]} \binom{t}{k}/\binom{t}{k-1}=\sup_{k \in [t/3,t]}\frac{t-k+1}{k} \le 3 .\]

 Putting everything together gives:
\[\mathbb{E}[Y_v(t-1)] \ge \sum_{\gamma' \in \hat \Gamma_{2t/3} } \sum_{\gamma \in  W_{ \gamma'}(t-1)}p(\gamma)Y_{\gamma(0)}(0) \ge \sfrac{1}{3} \sum_{\gamma' \in \hat \Gamma_{2t/3} } \sum_{\gamma \in  W_{ \gamma'}(t)}p(\gamma)Y_{\gamma(0)}(0)    . \]
 By \eqref{e:toolittlelazy}  the rightmost term above is at least $\sfrac{1}{3}(\mathbb{E}[Y_v(t)]-Y  \exp (-c_{1} t ) )$, as desired.
 \qed
\end{document}